\documentclass[11pt]{article}
\usepackage{amsmath, amsfonts, amsthm, amssymb, latexsym}
\usepackage{geometry}
\usepackage{enumerate}
\usepackage{multirow}
\usepackage{rotating}
\usepackage{graphicx,color}
\newcommand{\old}[1]{{}}
\newcommand\st{\mid}

\newcommand{\R}{\mathbb{R}}

\newcommand{\one}{\textbf{1}}
\newcommand{\zero}{\textbf{0}}

\newtheorem{thm}{Theorem}[section]
\newtheorem{lemma}[thm]{Lemma}

\newtheorem{prop}[thm]{Proposition}
\newtheorem{cor}[thm]{Corollary}

\theoremstyle{definition} \newtheorem{alg}{Algorithm}[section]

\title{On Chubanov's method for Linear Programming}
\author{Amitabh Basu, Jesus De Loera, Mark Junod  \\ Department of Mathematics, University of California, Davis}

\begin{document}

\maketitle

\begin{abstract}
We discuss the method recently proposed by S. Chubanov for the linear feasibility problem. We present new, concise proofs and interpretations of some of his results. 
We then show how our proofs can be used to find strongly polynomial time algorithms for special classes of linear feasibility problems. Under certain conditions, these 
results provide new proofs of classical results obtained by Tardos, and Vavasis and Ye.
\end{abstract}

\section{Introduction}

In their now classical papers, Agmon \cite{agmon} and Motzkin and
Schoenberg \cite{motzkinschoenberg} introduced the so called
\emph{relaxation method} to determine the feasibility of a system of
linear inequalities (it is well-known that optimization and the
feasibility problem are polynomially equivalent to one another).
Starting from any initial point, a sequence of points is generated.
If the current point $z_i$ is feasible we stop, else there must be at
least one constraint $c^Tx \leq d$ that is violated. We will denote the corresponding hyperplane by $(c,d)$. Let $p_{(c,d)}(z_i)$ be the orthogonal projection of $z_i$ onto the hyperplane $(c,d)$, choose a number
$\lambda$ (usually chosen between $0$ and $2$), and define the new point
$z_{i+1}$ by $z_{i+1}=z_i+\lambda(p_{(c,d)}(z_i)-z_i)$. Agmon, Motzkin and Schoenberg showed that if the original system of inequalities is feasible, then this procedure generates a sequence of points which converges, in the limit, to a feasible point. So in practice, we stop either when we find a feasible point, or when we are close enough, i.e., any violated constraint is violated by a very small (predetermined) amount.

 Many different versions of the relaxation method have been proposed,
 depending on how the step-length multiplier $\lambda$ is chosen and
 which violated hyperplane is used.  For example, the well-known perceptron
 algorithms~\cite{perceptron} can be thought of as members of this family of
 methods. In addition to linear programming feasibility, similar iterative ideas
 have been used in the solution of overdetermined system of linear
 equations as in the Kaczmarz's method where iterated projections into
 hyperplanes are used to generate a solution (see \cite{Kaczmarzoriginal,strohmervershynin,Needell}).

% A bad feature of the standard version of the relaxation
% method is when the system $Ax \leq b$ is infeasible, the algorithm
% cannot terminate because there will always be a violated
% inequality. 

The original relaxation method was shown early on to have a poor practical
convergence to a solution (and in fact, finiteness could only be
proved in some cases), thus relaxation methods took second place
behind other techniques for years.  During the height of the fame of
the ellipsoid method, the relaxation method was revisited with
interest because the two algorithms share a lot in common in their
structure (see \cite{amaldihauser,goffin,telgen} and references
therein) with the result that one can show that the relaxation method
is finite in all cases when using rational data, and thus can handle
infeasible systems. In some special cases the method did give a
polynomial time algorithm \cite{maurrasetal}, but in general it was
shown to be an exponential time algorithm (see \cite{goffinnonpoly,telgen}).  
Most recently in 2004, Betke gave a version that had polynomial guarantee 
in some cases and reported on experiments \cite{betkelp}.  In late 2010 
Sergei Chubanov presented a variant of the relaxation algorithm, that was based on the divide-and-conquer paradigm. We will refer to this algorithm as the {\em Chubanov Relaxation algorithm}.
The purpose of the Chubanov Relaxation algorithm \cite{chubanov} is to either 
find a solution of

\begin{equation} \label{eq:orsys}
\begin{array}{l}
Ax = b, \\
Cx \leq d
\end{array}
\end{equation}
\noindent
in $\R^n$, with $A$ an $m \times n$ matrix, $C$ an $l \times n$
matrix, $b \in \R^m$, and $d \in \R^l$, where the elements of $A$,
$b$, $C$, and $d$ are integers, or determine the system has no integer
solutions.  The advantage of Chubanov's algorithm is that when the inequalities take
the form $\textbf{0} \leq x \leq \textbf{1}$, then the algorithm runs
in \emph{strongly} polynomial time. This result can then be applied to give a
new polynomial time algorithm for linear optimization~\cite{chubanovlp}.  The purpose of
this paper is to investigate these recent ideas in the theory of linear optimization, simplify
some of his arguments, and show some consequences.
 
\subsection*{Our Results}

We start by explaining the basic details of Chubanov's algorithm in Section \ref{basicchubanov}; in particular, we outline the main ``Divide and Conquer'' subroutine from his paper. We will refer to this subroutine as {\em Chubanov's D\&C algorithm} in the rest of the paper. Chubanov's D\&C algorithm is the main ingredient in the Chubanov Relaxation algorithm. In the rest of Section \ref{basicchubanov}, we prove some key lemmas about the D\&C algorithm whose content can be summarized in the following theorem.

\begin{thm}\label{thm:summary}
Chubanov's D\&C algorithm can be used to infer one of the following statements about the system~\eqref{eq:orsys} :
\begin{itemize}
\item[(i)] A feasible solution to~\eqref{eq:orsys} exists, and can be found using the output of the D\&C algorithm.
\item[(ii)]  One of the inequalities in $Cx \leq d$ is an implied equality, i.e., there exists $k \in \{1, \ldots, l\}$ such that $c_k x = d$ for all solutions to~\eqref{eq:orsys}.
\item[(iii)] There exists $k \in \{1, \ldots, l\}$ such that $c_k x = d$ for all {\em integer} solutions to~\eqref{eq:orsys}.
\end{itemize}
\end{thm}

A constructive proof for the above theorem can be obtained using results in Chubanov's original paper~\cite{chubanov}. Our contribution here is to provide a different, albeit existential, proof of this theorem. We feel our proof is more geometric and simpler than Chubanov's original proof. We hope this will help to expose more clearly the main intuition behind Chubanov's D\&C algorithm, which is the workhorse behind Chubanov's results. Of course, Chubanov's constructive proof is more powerful in that it enables him to prove the following fascinating theorem.

\begin{thm}\label{thm:0-1sol}[see Theorem 5.1 in~\cite{chubanov}]
Chubanov's Relaxation algorithm either finds a solution to the system 
\begin{equation}\label{eq:0-1sys}\begin{array}{l}
Ax = b, \\
\mathbf{0} \leq x \leq \mathbf{1},
\end{array}
\end{equation}
or decides that there are no integer solutions to this system. Moreover, the algorithm runs in strongly polynomial time.
\end{thm}

This is an interesting theorem and leads to a new polynomial time algorithm for Linear Programming~\cite{chubanovlp}. However, it suffers from the drawback that it does not lead to a strongly polynomial time linear programming algorithm, even in the restricted setting of variables bounded between 0 and 1. Using the intuition behind our own proofs of Theorem~\ref{thm:summary}, we are able to demonstrate how Chubanov's D\&C algorithm can be used to give a strongly polynomial algorithm for deciding the feasibility or infeasibility of a system like~\eqref{eq:0-1sys}, under certain additional assumptions. In particular, we show the following result about bounded linear feasibility problems in Section~\ref{strictlps}.

\begin{thm}\label{thm:totally-unimodular} Consider the linear program given by 

\begin{equation}\label{eq:bddx-intro}\begin{array}{l}Ax = b, \\ \mathbf{0} \leq x \leq \lambda \mathbf{1}.\end{array}\end{equation} 

Suppose $A$ is a totally unimodular matrix and $\lambda$ is bounded by a polynomial in $n, m$ (the latter happens, for instance, when $\lambda=1$). Furthermore, suppose we know that if \eqref{eq:bddx-intro} is feasible, it has a strictly feasible solution. Then there exists a strongly polynomial time algorithm that either finds a feasible solution of \eqref{eq:bddx-intro}, or correctly decides that the system is infeasible. The running time for this algorithm is $O \left(m^3 + m^2n + n^2m + n^2(2n\lambda\sqrt{2n + 1})^{\frac{1}{\log_2\left(\frac{7}{5}\right)}}\right)$. If $\lambda=1$, then the running time can be upper bounded by $O (m^3 + m^2n + n^2m + n^{5.1})$.
\end{thm}

This theorem partially recovers E. Tardos' result on combinatorial LPs~\cite{tardos-strongly-poly}. Tardos' results were also obtained by Vavasis and Ye using interior point methods~\cite{yevavasis}, which is different from Tardos' approach. Our Theorem~\ref{thm:totally-unimodular} proves a weaker version of these classical results using a completely different set of tools, inspired by Chubanov's ideas.
Tardos' result is much stronger because she does not assume any upper bounds on the variables ($\lambda = \infty$), does not assume strictly feasible solutions, and only assumes that the entries of $A$ are polynomially bounded by $n,m$. Nevertheless our result has some interest as the techniques
are completely different from those in \cite{tardos-strongly-poly}, \cite{yevavasis}.

We also show that Chubanov's D\&C subroutine can be used to construct a new algorithm for solving general linear programs. 
This is completely different from Chubanov's linear programming algorithm in~\cite{chubanovlp}. However, the general algorithm 
that we present in this paper is not guaranteed to run in polynomial time. On the other hand, it has the advantage of avoiding some complicated reformulations that Chubanov uses to extract a general purpose LP algorithm from the Chubanov Relaxation algorithm. Moreover,  we can solve the problem in a single application of the D\&C subroutine, whereas the Chubanov Relaxation algorithm needs multiple applications. \old{These differences proved to be significant for practical purposes and we see that in practice our general purpose LP algorithm actually runs faster than Chubanov's algorithm. These, along with other experimental observations, are 
presented in Section~\ref{experiments}.}

\section{Chubanov's Divide and Conquer Algorithm} \label{basicchubanov}

In this section we will outline Chubanov's main subroutines for the D\&C algorithm as presented in \cite{chubanov}.

\par First, a couple of assumptions and some
notation.  We assume the matrices $A$ and $C$ have no zero rows and that $A$ is of
full rank.  Note if $A$ does not have full rank we can easily
transform the system into another, $A'x = b'$, $Cx \leq d$, such that
$A'$ has full rank without affecting the set of feasible solutions.
Let $a_i$ denote the $i$-th row of $A$ and $c_k$ denote the $k$-th row
of $C$.  $P$ will denote the set of feasible solutions of
\eqref{eq:orsys}.  Finally, $B(z,r)$ will denote the open ball
centered at $z$ of radius $r$ in $\R^n$.

One new idea of Chubanov's algorithm is its use of new induced inequalities.  Unlike
Motzkin and Schoenberg \cite{motzkinschoenberg}, who only projected onto the original
hyperplanes that describe the polyhedron $P$ (see top of Figure \ref{fig:newidea1}),
Chubanov constructs new valid inequalities along the way and projects onto them too (bottom part of Figure \ref{fig:newidea1}).

\begin{figure}[htb]
\centering
\includegraphics[width=7cm]{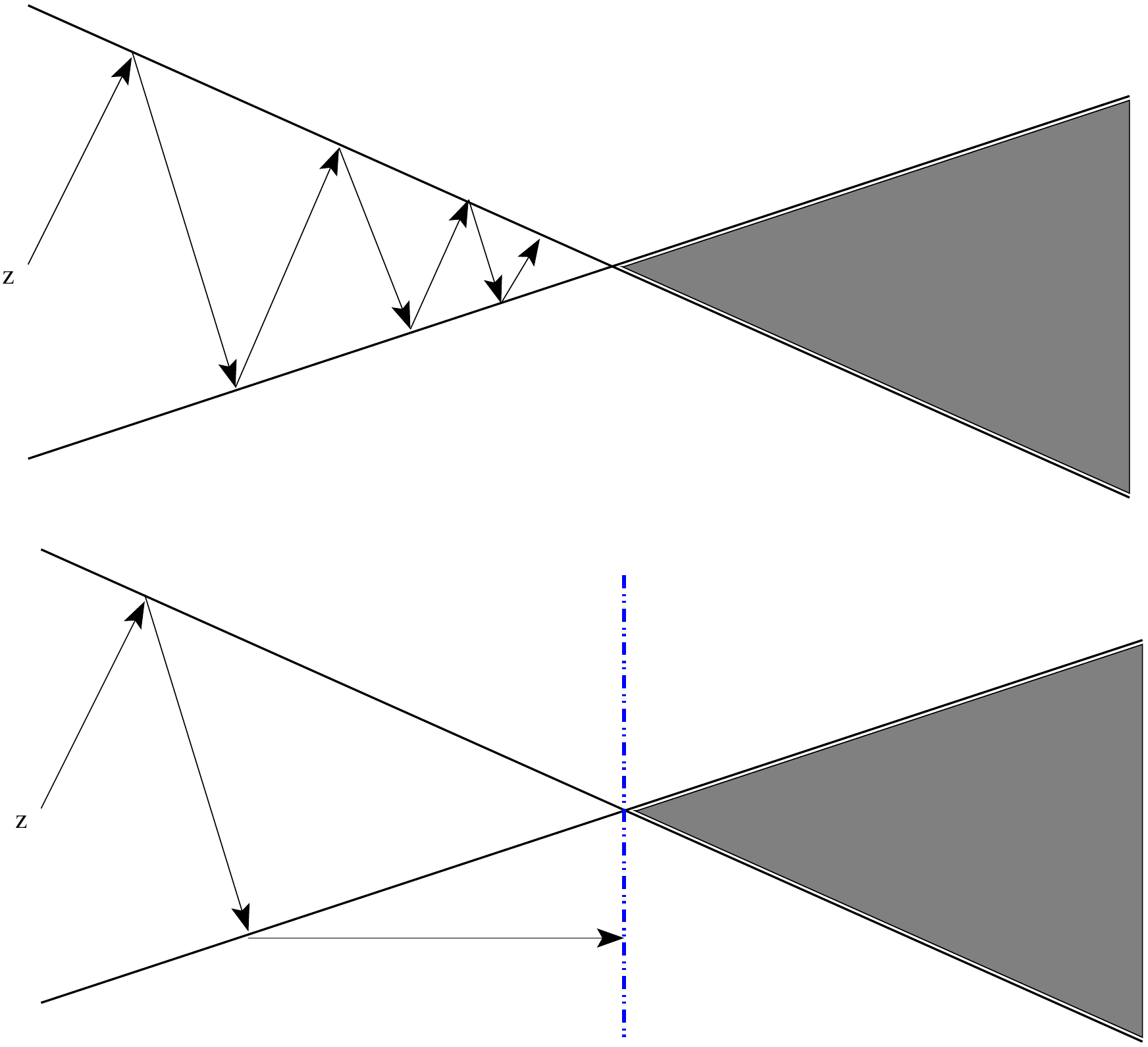}
\caption{Chubanov's method generates new inequalities on the way.}
\label{fig:newidea1}
\end{figure}
\old{
These new inequalities are linear combinations of the
input $a_i$'s and $b_i$'s and nonnegative linear combinations of the $c_k$'s
and $d_k$'s, called induced hyperplanes. Thus if $hx = \delta$ is an
induced hyperplane, then we have
$$h = \sum_{i = 1}^m \lambda_i a_i + \sum_{k = 1}^l \alpha_k c_k$$
and
$$\delta = \sum_{i = 1}^m \lambda_i d_i + \sum_{k = 1}^l \alpha_k d_k$$
where the $\lambda_i$'s and $\alpha_k$'s are real coefficients $\geq 0$.  
Exactly where these coefficients come from is explained later in this section. 
}

The aim of Chubanov's D\&C algorithm is to achieve the following. Given a
current guess $z$, a radius $r$ and an error bound $\epsilon > 0$, \old{the ball $B(z,r)$, }the algorithm will either:
\begin{enumerate}[(1)]
\item Find an $\epsilon$-approximate solution $x^* \in B(z,r)$ to \eqref{eq:orsys}, i.e. some $x^*$ such that 
$$Ax^* = b, \ Cx^* \leq d + \epsilon \one,$$ 
\item Or find an induced hyperplane $hx = \delta$ that separates $B(z,r)$ from $P$.
\end{enumerate}
This task is achieved using a recursive algorithm. In the base case, Chubanov uses a subroutine called \emph{Chubanov's Elementary Procedure} (EP), which achieves the above goal if $r$ is small enough; namely, when $r \leq \frac{\epsilon}{2 \|c_{\textrm{max}}\|}$, where $\|c_{\textrm{max}}\| = \max_{1 \leq k \leq l} (\|c_k\|)$. \old{Achieving this is initially the job of \emph{Chubanov's Elementary Procedure} (EP), which is called only when $r \leq \frac{\epsilon}{2 \|c_{\textrm{max}}\|}$, where $\|c_{\textrm{max}}\| = \max_{1 \leq k \leq l} (\|c_k\|)$.  The EP procedure has three distinct outcomes that solves this task.  }Let $p_{(A,b)}(z)$ denote the projection of $z$ onto the affine subspace defined by $Ax = b$.

\begin{alg}{\small THE ELEMENTARY PROCEDURE} \\
\emph{Input:} A system $Ax = b$, $Cx, \leq d$ and the triple $(z,r,\epsilon)$ where $r \leq \frac{\epsilon}{2 \|c_{\textrm{max}}\|}$. \\
\emph{Output:} Either an $\epsilon$-approximate solution $x^*$ or a separating hyperplane $hx = \delta$. \\
\\
{\bf If} $\|p_{(A,b)}(z) - z\| < r$ and $\frac{c_k z - d_k}{\|c_k\|} < r$ for all $k$ \\
\hspace*{.2in} {\bf Then} $x^* = p_{(A,b)}(z)$ is an $\epsilon$-approximate solution (see Figure \ref{fig:EP_Solution})\\
{\bf Else If} $\|p_{(A,b)}(z) - z\| \geq r$ \\
\hspace*{.2in} {\bf Then} let $h = (z - p_{(A,b)}(z))^T$ and $\delta = h \cdot p_{(A,b)}(z)$ (see Figure \ref{fig:EP_Affine}) \\
{\bf Else} $\frac{c_{k_0} z - d_{k_0}}{\|c_{k_0}\|} \geq r$ for some index $k_0$ \\
\hspace*{.2in} {\bf Then} let $h = c_{k_0}$ and $\delta = d_{k_0}$ (see Figure \ref{fig:EP_Inequals}) \\
{\bf End If}
\end{alg}

\begin{figure}[htbp]
\centering
\scalebox{0.8}{
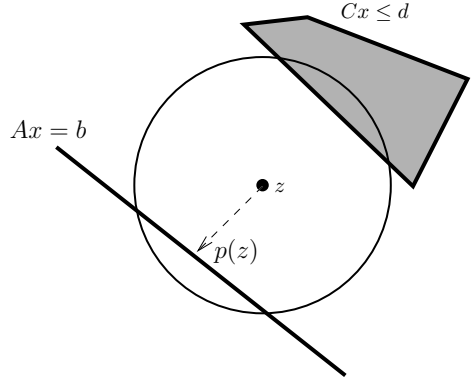
}
\caption{The projection $p(z)$ is an $\epsilon$-approximate solution.}
\label{fig:EP_Solution}
\end{figure}

\begin{figure}[htbp]
\centering
\scalebox{0.7}{
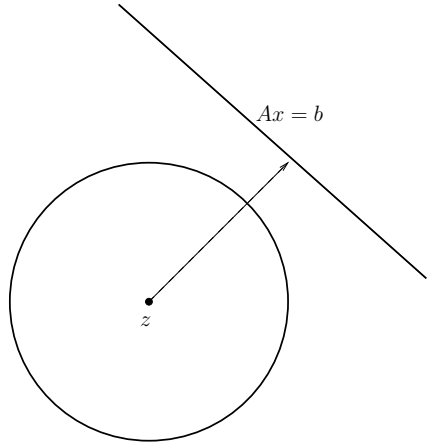
}
\caption{A separating hyperplane is given by the projection direction $(z-p(z))^T$.}
\label{fig:EP_Affine}
\end{figure}

\begin{figure}[htbp]
\centering
\scalebox{0.8}{
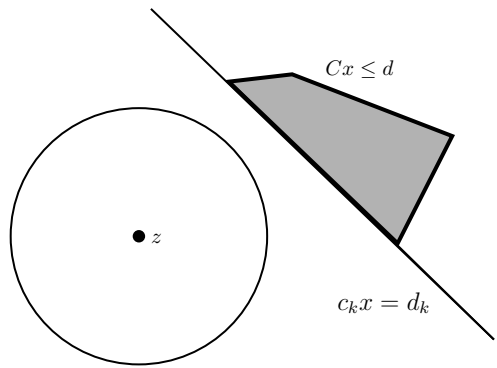
}
\caption{A separating hyperplane is given by a violated constraint.}
\label{fig:EP_Inequals}
\end{figure}

Note $\frac{c_k z - d_k}{\|c_k\|}$ tells us how far $z$ is from the
hyperplane $c_k x = d_k$, and it is in fact negative if $z$ satisfies the
inequality $c_k x \leq d_k$.  Thus if $\frac{c_k z - d_k}{\|c_k\|} <
r$, then any point in $B(z,r)$ is an $\epsilon$-approximation of $c_k
x \leq d_k$. This simple observation is enough to see that the EP procedure solves the task when $r \leq \frac{\epsilon}{2
  \|c_{\textrm{max}}\|}$. \old{So if (i) holds, then it follows $p(z)$ solves the first
task.  The other outcomes, (ii) and (iii), solve the second task with
similar logic. } See Chubanov \cite{chubanov} Section 2 for more details
and proofs. \par

The elementary procedure achieves the goal when $r$ is small enough, but what happens when $r > \frac{\epsilon}{2\|c_{\textrm{max}}\|}$? Then the D\&C Algorithm makes recursive steps with smaller values of $r$. To complete these recursive steps, the D\&C algorithm uses additional projections and separating hyperplanes. We give more details below. Figure~\ref{fig:two_circles} illustrates some of the steps in this recursion. A sample recursion tree is shown in Figure \ref{fig:recursion_tree}.

\begin{figure}[htbp]
\centering
\scalebox{0.5}{
\includegraphics{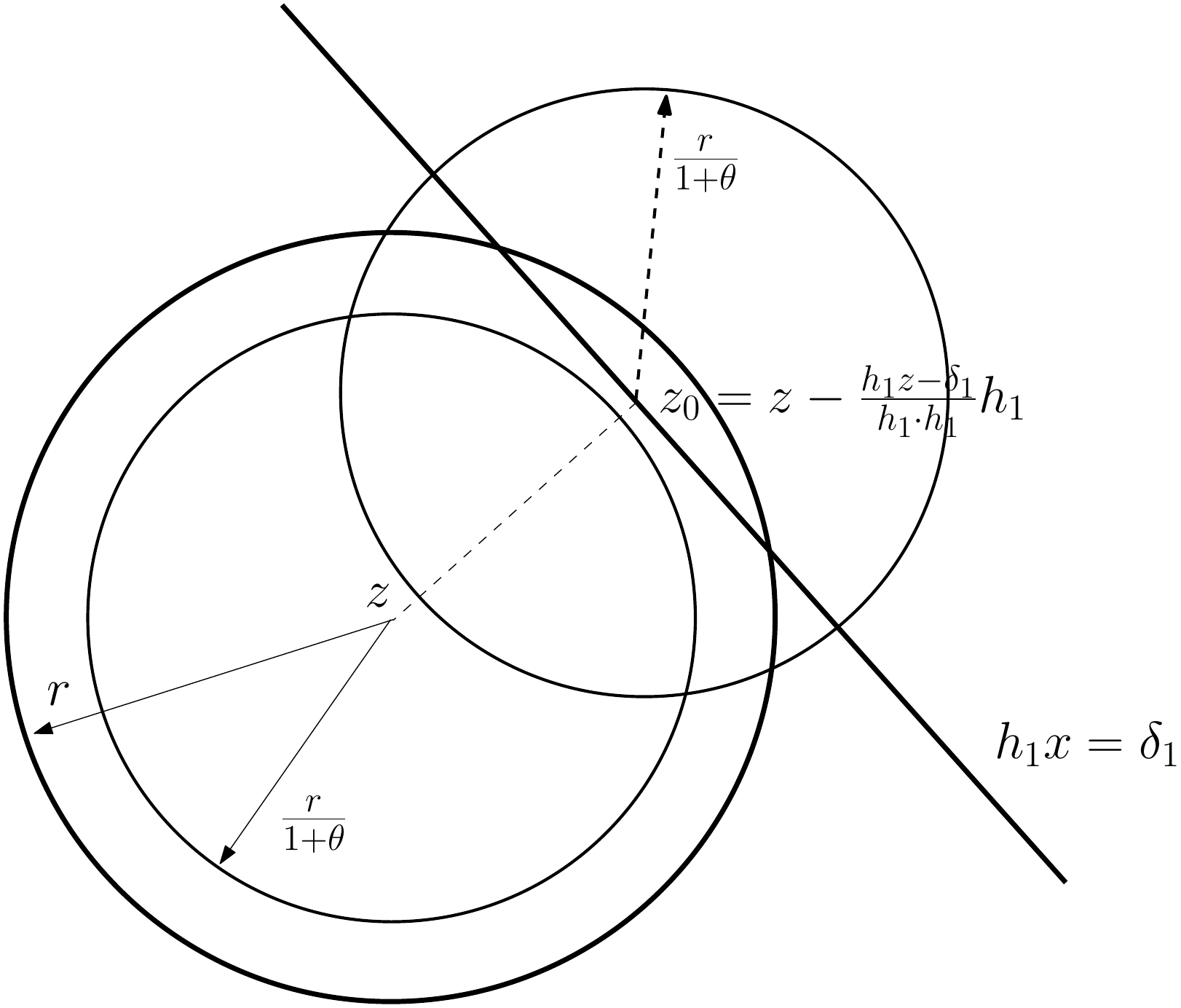}
}
\caption{The recursive step in D\&C works on two smaller balls with centers $z$ and $z_0$}
\label{fig:two_circles}
\end{figure}

\old{However, note the EP works only when we have $r \leq \frac{\epsilon}{2
  \|c_{\textrm{max}}\|}$.  When $r$ is larger than this, we rely on 
  \emph{Chubanov's Divide and Conquer Algorithm} (D\&C), which 
calls the EP as a subroutine in order to help solve the two tasks.  At
its most basic, if $r$ is small enough, the D\&C simply calls the
Elementary Procedure.  Otherwise the algorithm recursively finds
smaller and smaller $r$ until the EP can be applied.  The D\&C then
continues to test for both $\epsilon$-approximate solutions $x^*$ and
separating hyperplanes as it works its way back up through the
recursion (see Figure \ref{fig:recursion_tree}).} \par

\begin{alg} {\small THE D\&C ALGORITHM} \\
\emph{Input:} A system $Ax = b$, $Cx \leq d$ and the triple $(z, r, \epsilon)$. \\
\emph{Output:} Either an $\epsilon$-approximate solution $x^*$, or a separating hyperplane $hx = \delta$. \\

\noindent \textbf{If} $r \leq \frac{\epsilon}{2 \| c_{\textrm{max}} \|}$ \\
\hspace*{.2in} \textbf{Then} run the EP on the system \\
\textbf{Else} recursively run the D\&C Algorithm with $(z, \frac{1}{1 + \theta}r, \epsilon)$ \\
{\bf End If} \\

\noindent \textbf{If} the recursive call returns a solution $x^*$ \\
\hspace*{.2in} \textbf{Return} $x^*$ \\
\textbf{Else} let $h_1 x = \delta_1$ be returned by the recursive call \\
{\bf End If} \\

\noindent \textbf{Set} $z_0 = z - \frac{h_1 z - \delta_1}{h_1 \cdot h_1} \cdot h_1$, i.e., project $z$ onto $(h_1, \delta_1)$ (see Figure~\ref{fig:two_circles})\\
\textbf{Run} the D\&C Algorithm with $(z_0, \frac{1}{1 + \theta}r, \epsilon)$ \\

\noindent \textbf{If} the recursive call returns a solution $x^*$ \\
\hspace*{.2in} \textbf{Return} $x^*$ \\
\textbf{Else} let $h_2 x = \delta_2$ be returned by the recursive call \\
{\bf End If} \\

\noindent \textbf{If} $h_1 = -\gamma h_2$ for some $\gamma > 0$ \\
\hspace*{.2in} \textbf{Then} STOP, the algorithm fails \\
\textbf{Else} Find $\alpha$ such that $h = \alpha h_1 + (1-\alpha)h_2$, $\delta = \alpha\delta_1 + (1-\alpha)\delta_2$, and $\frac{hz - \delta}{\|h\|} \geq r$ \\
\hspace*{.2in} \textbf{Return} $hx = \delta$
{\bf End If} \\
\end{alg}

\begin{figure}[htpb]
\centering
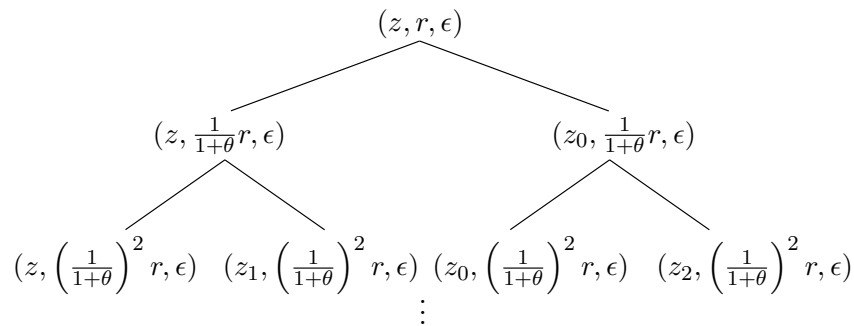
\caption{D\&C Recursion Tree}
\label{fig:recursion_tree}
\end{figure}

\old{
There are a couple things we need to note about the above algorithm:
\begin{enumerate}
\item Only the EP can truly return an $\epsilon$-approximate solution.  In the higher levels of the recursion process, when $r > \frac{\epsilon}{2 \| c_{\textrm{max}} \|}$ , the D\&C only returns new induced separating hyperplanes. 
\item In the higher levels of the recursion, we are really only dealing with some $h_1$ and $h_2$ returned from the previous level (see figure {\bf MISSING, DO WE HAVE ALL FIGURES FROM HIS SLIDES?}).  This takes place exclusively in the second to final \textbf{Else} statement of Algorithm 2.1.  
\item The induced hyperplanes are not added in any way to the system \eqref{eq:orsys}.  As soon as the D\&C moves beyond that level of the recursion, the induced hyperplanes no longer need to be held in memory.  

\item If the original system is infeasible, then no solution $x^*$ will be returned, even if the $\epsilon$-approximate system is feasible.

\item Chubanov has ways to deal with the growth of the numbers used in the D\&C algorithm and proves its correctness.  We
skip these and other details here, including the running time of the algorithm, see Section three of \cite{chubanov}. \par

\item Note that there is a strange out come in the D\&C labeled as ``fails''. This occurs when it 
returns $h_1$ and $h_2$ such that $h_1 = -\gamma h_2$. This is a gray area that does not give a desire outcome and
needs to be addressed somehow.

\end{enumerate}
}

We now state the running time of the Chubanov D\&C Algorithm. \old{We first define certain parameters depending on the input matrices $A$ and $C$. Let $\mu = \frac{2\epsilon}{28n\lVert c_{max}\rVert^2}$ where $c_{max}$ is the row of $C$ with maximum norm. Let $\rho$ be a number such that $|z_j| \leq \rho$ and $\rho z_j$ is an integer for all components $z_j$ of the center $z$ in the input to the D\&C algorithm. The matrix $A$ is assumed to have $m$ rows and $n$ columns and let $N$ be the number of non-zero entries in the matrix $C$. Let $K = \big(\frac{r\lVert c_{max}\rVert}{\epsilon}\big)^\frac{1}{log_2 (7/5)}$.}

\begin{prop}\label{prop:run-time}[Theorem 3.1 in \cite{chubanov}]
The matrix $A$ is assumed to have $m$ rows and $n$ columns and let $N$ be the number of non-zero entries in the matrix $C$. Let $\mu = \frac{2\epsilon}{28n\lVert c_{max}\rVert^2}$ where $c_{max}$ is the row of $C$ with maximum norm. Let $\rho$ be a number such that $|z_j| \leq \rho$ and $\rho z_j$ is an integer for all components $z_j$ of the center $z$ in the input to the D\&C algorithm. Let $K = \big(\frac{r\lVert c_{max}\rVert}{\epsilon}\big)^\frac{1}{log_2 (7/5)}$.

Chubanov's D\&C algorithm performs at most \begin{equation}\label{eq:run-time}O(m^3 + m^2n + n^2m + K\big(n\log(\frac{\rho + log(K)(r+n\mu)}{\mu}) + n^2 + N\big))\end{equation} arithmetic operations.
\end{prop}

\subsection{Using the D\&C Algorithm}

One possible way to exploit the D\&C algorithm is the following. Since D\&C returns $\epsilon$-approximate solutions, we can try to run it on the system $Ax=b, Cx \leq d - \epsilon\mathbf{1}$; if the algorithm returns an $\epsilon$-approximate solution, we will have an {\em exact} solution for our original system $Ax=b, Cx \leq d$. However, we need a $z$ and an $r$ as input for D\&C. To get around this, we can appeal to some results from classical linear programming theory. Suppose one can, a priori, find a real number $r^*$ such that if the system $Ax=b,\; Cx \leq d - \epsilon\one$ is feasible, then it has a solution with norm at most $r^*$. In other words, there exists a solution in $B(0,r^*)$, if the system is feasible. Such bounds are known in the linear programming literature (see Corollary 10.2b in Schrijver~\cite{Schrijver1986}), where $r^*$ depends on the entries of $A, b, C, d$ and $\epsilon$. Then one can choose $z=0$ and $r = r^*$ for the D\&C algorithm. If the algorithm returns an $\epsilon$-approximate solution, we will have an {\em exact} solution for our original system $Ax=b, Cx \leq d$; whereas, if the algorithm returns a separating hyerplane, we know that $Ax=b, Cx \leq d - \epsilon\mathbf{1}$ is infeasible by our choice of $r$. \par

This strategy suffers from three problems.
 
\begin{enumerate}
\item There is a strange outcome in the D\&C procedure - when it ``fails" and stops. This occurs when the two recursive branches return hyperplanes with normal vectors $h_1, h_2$ with $h_1 = -\gamma h_2$ for some $\gamma > 0$. It is not clear what we can learn about the problem from this outcome. Later in this paper, we will interpret this outcome in a manner that is different from Chubanov's interpretation. 
\item It might  happen that $Ax=b, Cx \leq d - \epsilon\mathbf{1}$ is infeasible, even if the original system $Ax=b, Cx \leq d$ is feasible. In this case the algorithm may return a separating hyperplane, but we cannot get any information about our original system. All we learn is that $Ax=b, Cx \leq d - \epsilon\mathbf{1}$ is infeasible.
\item Finally, the running time of the D\&C algorithm is a polynomial in $n,m$ and $r = r^*$. Typically, the classical bounds on $r^*$ are exponential in the data. This would mean the running time of the algorithm is not polynomial in the input data. 
\end{enumerate}

Let us concentrate on tackling the first two problems, to progress towards a correct linear programming algorithm. Then we can worry about the running time. This requires a second important idea in Chubanov's work. We address the first two problems above by homogenizing, or parameterizing, our original system, and we show how this helps in the rest of this section. Geometrically this turns the original polyhedron into an unbounded polyhedron, defined by
\begin{equation} \label{eq:parsys}
\begin{array}{l}
Ax - bt = \zero, \\
Cx - dt \leq \zero, \\
-t \leq -1.
\end{array}
\end{equation}
Note this system \eqref{eq:parsys} is feasible if and only if \eqref{eq:orsys} is feasible.  Let $(x^*, t^*)$ be a solution to \eqref{eq:parsys}.  Then $\frac{x^*}{t^*}$ is a solution of \eqref{eq:orsys}.  Similarly, if $x^*$ is a solution of \eqref{eq:orsys}, then $(x^*, 1)$ is a solution of \eqref{eq:parsys}.  Thus we can apply the D\&C to a strengthened parameterized system
\begin{equation} \label{eq:strsys}
\begin{array}{l}
Ax - bt = \zero, \\
Cx - dt \leq -\epsilon \one, \\
-t \leq -1 - \epsilon
\end{array}
\end{equation}
and any $\epsilon$-approximate solution will be an exact solution of \eqref{eq:parsys}, and thus will give us an exact solution of \eqref{eq:orsys}. We still need to figure out what $z$ and $r$ to use. For the rest of the paper, we will use $\epsilon = 1$ as done by Chubanov in his paper. It turns out that if we choose the appropriate $z$ and $r$, we can get interesting information about the original system $Ax=b, Cx \leq d$ when D\&C fails, or returns a separating hyperplane. We explain this next.
\old{
Now we still have the same issues as before, that is \eqref{eq:strsys} might be infeasible even though \eqref{eq:parsys} is feasible.  However, because we are now dealing with a feasible region that is unbounded and essentially a cone, we can gain insights from both infeasible systems and systems where the D\&C fails that we could not from the non-parameterized system.  Two lemmas important from Chubanov \cite{chubanov} give us what we need. \par

\begin{lemma}[Chubanov Lemma 4.1]
Let the D\&C when applied to \eqref{eq:strsys} fail (i.e., $h_1 = -\gamma h_2$) and suppose the original system \eqref{eq:orsys} is feasible.  Let
$$h_j = a^j + \sum_{k = 1}^l \alpha_k^j(c_k, -d_k) - \beta^j (\zero, 1),$$
for $j = 1, 2$.  Then
\begin{enumerate}[(1)]
\item $\beta^1 = 0$ and $\beta^2 = 0$.
\item At least one of the coefficients $\alpha_k^j$ is nonzero.  If $\alpha_{k_0}^j > 0$ for $j=1$ or $j=2$, then $c_{k_0} x = d_{k_0}$ for all $x \in P$.
\end{enumerate}
\end{lemma}

\noindent Now assume $r^*$ is a positive number such that if $P$ is non-empty, then $P \subset B(\zero, r^*)$.

\begin{lemma}[Chubanov Lemma 4.2]
Let the D\&C algorithm return the induced hyperplane $hx = \delta$ when applied to \eqref{eq:strsys} with input parameters $(x^0, t^0) = (\zero, 0)$, $r = (2l + 3)(r^* + 1)$, and $\epsilon = 1$.
\begin{enumerate}[(1)]
\item If $\beta \geq \max_k \alpha_k$, then $\|x\| \geq r^*$ for all $x \in P$.
\item If $\alpha_{k_0} = \max_k \alpha_k$, and $\alpha_{k_0} \geq \beta$, then $d_{k_0} - c_{k_0} x \leq \frac{1}{2}$ for all $x \in P$ such that $\|x\| < r^*$.
\end{enumerate}
\end{lemma}

\noindent For full proofs and details, again see \cite{chubanov} Section 4.  \par

The purpose of the Chubanov Relaxation Algorithm \cite{chubanov} is to either 
find a real solution of  \ref{eq:orsys} or determine the system has no integer
solutions. We put together the principles we have outlined.
As a quick note, for the rest of the paper, we are going to assume $\epsilon = 1$ since we no longer care about the exact value of $\epsilon$ as we will always be finding exact solutions or showing the original system (not the $\epsilon$-approximate system) is in some way infeasible. 
}

Let us summarize before we proceed. Given a system \eqref{eq:orsys} we parameterize and then strengthen with $\epsilon = 1$ (to obtain a system in the form \eqref{eq:strsys}).\old{ we can break Chubanov's algorithm down into three different, possibly recurring, cases.  First we need to make sure $Ax = b$ is feasible, else no solution.  } Then we apply the D\&C to \eqref{eq:strsys} with appropriately chosen $z$ and $r$.  Our three possible outcomes are:

\begin{enumerate}[(I)]
\item The D\&C gives us a solution $(x^*, t^*)$ which is an $\epsilon=1$-approximate solution to \eqref{eq:strsys}. This is the best possible outcome, because we can then return the exact solution $\frac{x^*}{t^*}$ to \eqref{eq:orsys}.
\item The D\&C fails. The reader can look ahead to our Proposition~\ref{prop:eq-forced} for an interpretation of this outcome. \old{ Then we can apply Lemma 2.1.  If either of the $\beta^j$'s is nonzero, or if all the $\alpha_k^j$'s are zero, then we know the system \eqref{eq:orsys} is infeasible, and we can return the fact no \emph{integer} solutions exist.  Otherwise we know one of the inequalities is really an implied equality, and we can add $c_{k_0} x = d_{k_0}$ to $Ax = b$.  Then run the D\&C on the new system after paramaterizing and strengthening $A'x = b'$, $C'x \leq d$ where the implied inequality has been moved to the equalities.}
\item The D\&C returns a separating hyperplane $hx = \delta$. Our Proposition~\ref{prop:thin-strip} tells us how to interpret this outcome. \old{Then we can apply Lemma 2.2.  Again, either the original system has no integer solutions, or according to statement (2) of Lemma 2.2, any integer solutions must lie on $c_{k_0} x = d_{k_0}$, and so we get an implied equality which we can move to $Ax = b$.  Then run again the D\&C on the new (parameterized and strengthened) system.}
\end{enumerate}
\old{
First note this process must stop since there are only a finite number of inequalities to begin with, and only a finite number of 
inequalities we can turn into equalities.  Second, only if we determine the system is infeasible during the first iteration can we actually claim \eqref{eq:orsys} is completely infeasible.  During any other iteration we just know it cannot have any \emph{integer} solutions as outcome (III) might have turned a feasible system with no integer solutions into an infeasible system.  At the same time, just because we find a feasible solution $x^*$ of \eqref{eq:orsys}, unless $x^*$ is itself an integer solution, there is still no guarantee \eqref{eq:orsys} has any integer solutions.  Thus in and of itself, the Chubanov Relaxation Algorithm only solves the feasibility question or the integer feasibility question in very specific circumstances.  
The main strength of the algorithm lies in the fact that if the system takes the form $Ax = b$, $\zero \leq x \leq \one$, then it runs in strongly polynomial time. %\par {\bf HEY MAYBE HERE WE CAN ADD THE PRECISE RUNNING TIME???}
 }

%%%%%%%%%%%%%%%%%%%%%%%%%%%%%%%%%%%%%%%%%%%%%%%%%%%%%%%%%%%
%%%%%%%%%%%%%%%%%%%%%%%%%%%%%%%%%%%%%%%%%%%%%%%%%%%%%%%%%%%

%\subsection{Geometric intuition on the key Chubanov Lemmas}

In the rest of this section, we give some more geometric intuition behind the process of homogenizing the polyhedron and why it is useful. Our goal will be to prove Theorem~\ref{thm:summary}.

\subsection{Meaning of ``Failure" Outcome in D\&C}
First we show that if the Chubanov D\&C algorithm fails on a particular system, then in fact that system is infeasible. This observation is never made in the original paper by Chubanov~\cite{chubanov} and, as far as we know, is new.

\begin{prop}\label{prop:failure1}  Suppose Chubanov's D\&C Algorithm fails on the system $Ax=b, Cx \leq d$, i.e., it returns two hyperplanes $h_1x = \delta_1$ and $h_2x = \delta_2$ with $h_1 = -\gamma h_2$ with $\gamma > 0$. Then the system $Ax=b, Cx \leq d$ is infeasible.
\end{prop}

\begin{proof}
Let $P = \{x \in \R^n \st Ax=b, Cx \leq d\}$. If Chubanov's algorithm fails, then there exists $z \in \R^n, r > 0$ such that the following two things happen :

\begin{enumerate}
\item[(i)] $h_1x \leq \delta_1$ is valid for $P$ and for all $y \in B(z, \frac{1}{1+\theta}r)$, $h_1y > \delta_1$.
\item[(i)] $h_2x \leq \delta_2$ is valid for $P$ and for all $y \in B(z_0, \frac{1}{1+\theta}r)$, $h_2y > \delta_2$, where $z_0 = z - \frac{h_1z - \delta_1}{h_1\cdot h_1}\cdot h_1$.
\end{enumerate}

Since $z_0 - \frac{r}{2(1+\theta)}\frac{h_2}{\lVert h_2\rVert} \in B(z_0, \frac{1}{1+\theta}r)$, $h_2\cdot(z_0 - \frac{r}{2(1+\theta)}\frac{h_2}{\lVert h_2\rVert}) > \delta_1$. Therefore,
\old{\begin{equation}\label{eq:h1}
h_1z - \frac{r\lVert h_1\rVert}{2(1 + \theta)} > \delta_1.
\end{equation}
Using the same argument with $B(z_0,  \frac{1}{1+\theta}r)$ and $h_2, \delta_2$, we get that }
\begin{equation}\label{eq:h2}
h_2z_0 - \frac{r\lVert h_2\rVert}{2(1 + \theta)} > \delta_2.
\end{equation} 

Now we use the fact that $h_1 = -\gamma h_2$ and so $-\frac{1}{\gamma}h_1 = h_2$ which we substitute into \eqref{eq:h2} to get $-\frac{1}{\gamma}h_1z_0 - \frac{r\lVert h_2\rVert}{2(1 + \theta)} > \delta_2$. From the definition of $z_0$, we have $h_1z_0 = \delta_1$.  Therefore, $-\frac{1}{\gamma}\delta_1  > \frac{r\lVert h_2\rVert}{2(1 + \theta)} + \delta_2 > \delta_2$. So $\delta_1 < -\gamma\delta_2$. Now $h_1x \leq \delta_1$ is valid for $P$ using (i) above. Using $h_1 = -\gamma h_2$ and $\delta_1 < -\gamma \delta_2$, we get $-\gamma h_2 x < -\gamma\delta_2$ is valid for $P$, i.e., $h_2x > \delta_2$ is valid for $P$. But we also know that $h_2x \leq \delta_2$ is valid for $P$ from (ii) above. This implies that $P = \emptyset$ and the system $Ax=b, Cx \leq d$ is infeasible.\end{proof}

%\subsection{Failure of D\&C on strengthened parameterized system}
We now interpret the ``failure" outcome of the D\&C algorithm on the strengthened parameterized system. First we prove the following useful lemma.

\begin{figure}
\centering
\scalebox{0.5}{\includegraphics{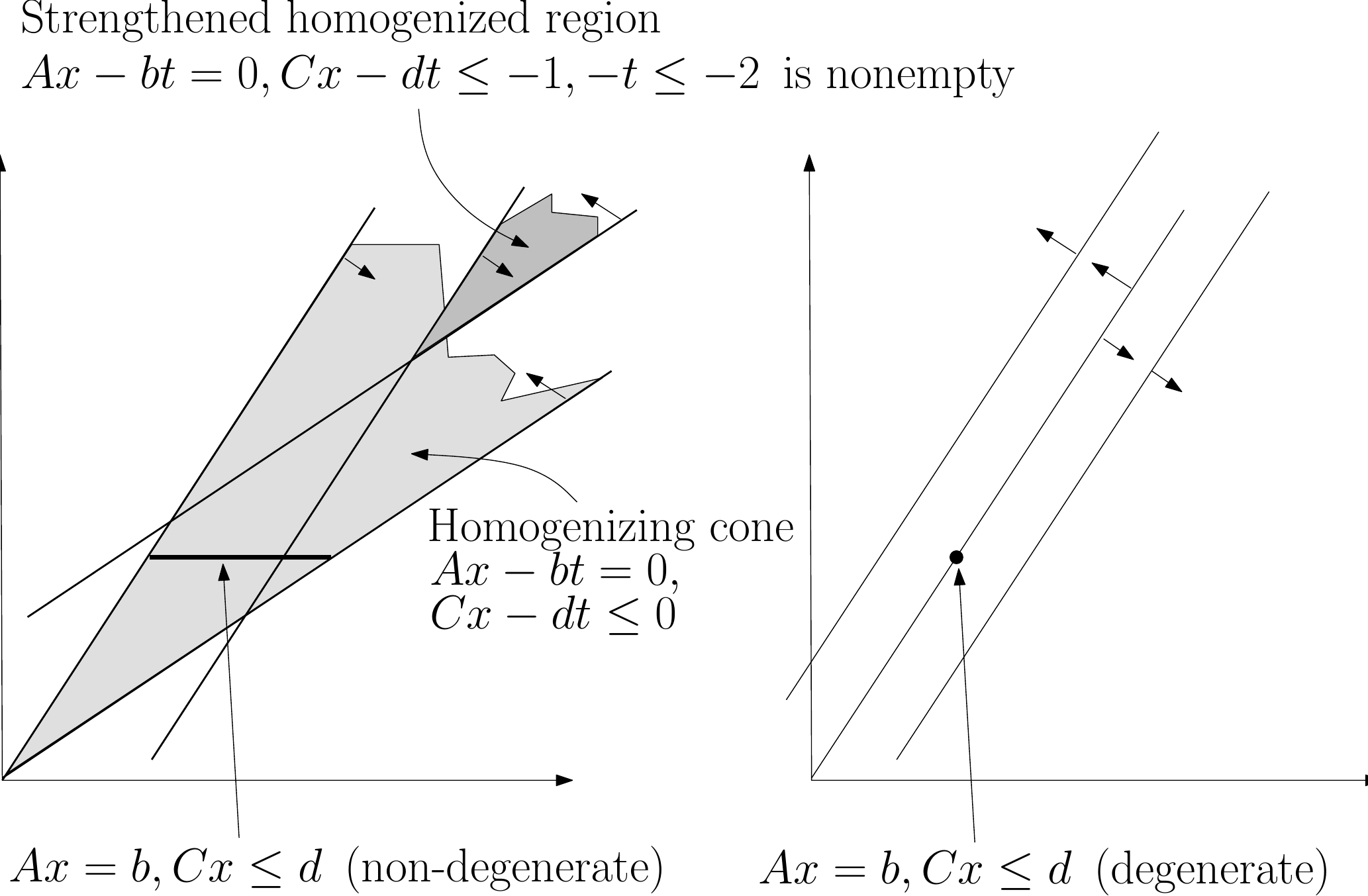}}
\caption{This Figure illustrates our Lemma~\ref{lem:equality-forcing1}. In the left figure, the original feasible region is non-degenerate and so the strengthened, homogenized cone is nonempty. The figure on the right shows an example where the original system is degenerate; the strengthened cone is empty because we have two parallel hyperplanes which get pushed away from each other, creating infeasibility in the strengthened system. }\label{fig:homogenization_1}
\end{figure}
 
\begin{lemma}\label{lem:equality-forcing1}
If the system $Ax - bt = 0, Cx - dt \leq -1, -t \leq -2$ is infeasible, then there exists $l \in \{1, \ldots, k\}$ such that $c_l\cdot x = d_l$ for all $x$ satisfying $Ax=b, Cx \leq d$.
\end{lemma}

\begin{proof}
We prove the contrapositive. So suppose that for all $k \in \{1, \ldots, l\}$, there exists $x_k$ satisfying $Ax_k=b, Cx_k \leq d$ with $c_k\cdot x_k < d_k$. Then using $\bar x = \frac{1}{l}(x_1 + \ldots + x_l)$, we get that $A\bar x = b$ and $c_k\bar x < d_k$ for all $k \in \{1, \ldots, l\}$. Let $\eta_k = d_k - c_k\cdot\bar x > 0$ and let $\eta = \min\{\frac{1}{2}, \eta_1, \ldots, \eta_l\}$. Therefore, $\eta > 0$. Let $x^* = \frac{\bar x}{\eta}$ and $t^* = \frac{1}{\eta}$. Then 
$$Ax^* - bt^* = \frac{1}{\eta}(A\bar x - b) = 0.$$
For every $k \in \{1, \ldots, l\}$, 
$$d_k t^* - c_k x^* = \frac{1}{\eta}(d_k - c_k \bar x) = \frac{\eta_k}{\eta}\geq 1.$$ 
Therefore, $c_k x^* - d_k t^* \leq -1$ for every $k \in \{1, \ldots, l\}$. Finally, since $t = \frac{1}{\eta} \geq 2$ by definition of $\eta$, we have $-t \leq -2$. \end{proof}

An illustration of the above lemma appears in Figure~\ref{fig:homogenization_1}. The following is the important conclusion one makes if the D\&C algorithm ``fails" on the strengthened parameterized system.

\begin{prop}\label{prop:eq-forced}
If Chubanov's D\&C algorithm fails on the system $Ax - bt = 0, Cx - dt \leq -1, -t \leq -2$, then there exists $l \in \{1, \ldots, k\}$ such that $c_l\cdot x = d_l$ for all $x$ satisfying $Ax=b, Cx \leq d$.
\end{prop}
\begin{proof}Follows from Proposition~\ref{prop:failure1} and Lemma~\ref{lem:equality-forcing1}.\end{proof}

\subsection{The meaning of when a separating hyperplane is returned by D\&C}

We now make the second useful observation about the strengthened parameterized system $Ax - bt = 0, Cx - dt \leq -1, -t \leq -2$. Suppose we know that all solutions to $Ax=b, Cx \leq d$ have norm at most $r^*$. Then we will show that if all solutions to the strengthened system are ``too far" from the origin, then the original system is ``very thin" (this intuition is illustrated in Figure~\ref{fig:homogenization_2}). More precisely, we show that if all solutions to the strengthened parameterized system have norm greater than $2k(r^* + 1)$, then there exists an inequality $c_lx\leq d_l$ such that all solutions to $Ax=b, Cx \leq d$ satisfy $d_l - \frac{1}{2} \leq c_l x$. That is, the original polyhedron lies in a ``thin strip" $d_l - \frac{1}{2} \leq c_l x \leq d_l$. This would then imply that all integer solutions to $Ax=b, Cx \leq d$ satisfy $c_lx = d_l$ since $c_l, d_l$ have integer entries. Here is the precise statement of this observation.

\begin{figure}
\centering
\scalebox{0.55}{\includegraphics{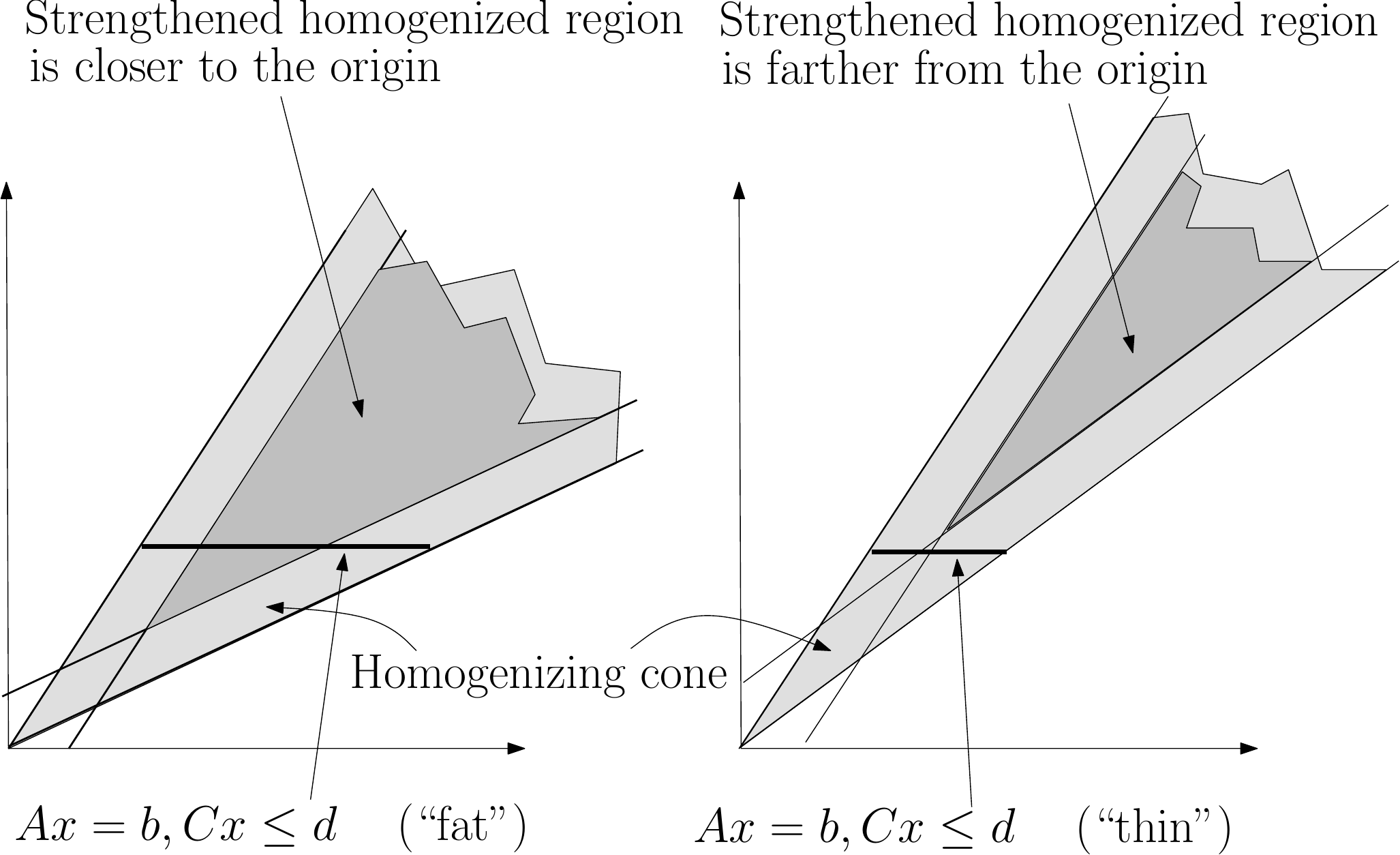}}
\caption{This Figure illustrates our Proposition~\ref{prop:thin-strip}. When the original feasible region is ``thinner'', the strengthened homogenized cone is pushed farther away from the origin. }\label{fig:homogenization_2}
\end{figure}

\begin{prop}\label{prop:thin-strip}
Suppose $r^* \in \R$ is such that $\lVert x \rVert \leq r^*$ for all $x$ satisfying $Ax = b, Cx \leq d$. If $\lVert (x,t) \rVert > 2k(r^* + 1)$ for all $(x,t)$ satisfying $Ax - bt = 0, Cx - dt \leq -1, -t \leq -2$, then there exists $l \in \{1, \ldots, k\}$ such that $d_l - \frac{1}{2} \leq c_l x \leq d_l$ for all $x$ satisfying $Ax = b, Cx \leq d$.
\end{prop}

\begin{proof} Suppose to the contrary that for all $j \in \{1, \ldots, k\}$, there exists $x^j$ such that $Ax^j = b, Cx^j \leq d$, and $c_jx^j \leq d^j - \frac{1}{2}$, i.e., $c_j(2x^j) - 2d_j \leq -1$. This implies that the following equations hold for all $j \in \{1, \ldots, k\} $\begin{equation}\label{eq:eqs} \begin{array}{l} c_j(2x^j) - 2d_j \leq -1, \\ c_j(2x^i) - 2d_j \leq 0 \qquad \forall i \neq j.\end{array}    \end{equation}  Now consider $\hat x = \sum_{j=1}^k 2x^j$ and $\hat t = 2k$. It is easily verified $A\hat x - b\hat t = 0$ since $Ax^j = b$ for all $j \in \{1, \ldots, k\}$. Moreover, adding together the inequalities in \eqref{eq:eqs}, we get that $c_j\hat x - d_j\hat t \leq -1$ for all $j \in \{1, \ldots, k\}$.  Therefore, $(\hat x, \hat t)$ satisfies the constraints $Ax - bt = 0, Cx - dt \leq -1, -t \leq -2$. 

Finally, $\lVert (\hat x, \hat t) \rVert \leq \lVert \hat x\rVert + 2k \leq \sum_{j=1}^k 2\lVert x^j\rVert + 2k \leq 2k(r^* + 1)$, where the last inequality follows from the fact that all solutions to $Ax = b, Cx \leq d$ have norm at most $r^*$. We have thus reached a contradiction with the hypothesis that $\lVert (x,t) \rVert > 2k(r^* + 1)$ for all $(x,t)$ satisfying $Ax - bt = 0, Cx - dt \leq -1, -t \leq -2$.\end{proof}

The above proposition shows that if Chubanov's D\&C algorithm returns a separating hyperplane separating the feasible region of $Ax - bt = 0, Cx - dt \leq -1, -t \leq -2$ from the ball $B(0, 2k(r^*+1))$, one can infer that there exists an inequality $c_l x \leq d_l$ that is satisfied at equality by all integer solutions to $Ax=b, Cx \leq d$. 

We now have all the tools to prove Theorem~\ref{thm:summary}.

\begin{proof}[Proof of Theorem~\ref{thm:summary}] 
As discussed earlier, we can assume that there exists $r^*$ such that $\lVert x \rVert \leq r^*$ for all $x$ satisfying $Ax = b, Cx \leq d$. We then run the D\&C algorithm on the strengthened system~\eqref{eq:strsys} with $\epsilon=1$, $z=0$ and $r=2k(r^*+1)$. Recall the three possible outcomes of this algorithm. In the first outcome, we can find an exact solution of $Ax=b, Cx \leq d$ - this is (i) in the statement of the theorem. In the second outcome, when D\&C fails, Proposition~\ref{prop:eq-forced} tells that we have an implied equality, which is (ii) in the statement of the theorem. Finally, in the third outcome, D\&C returns a separating hyperplane. By Proposition~\ref{prop:thin-strip}, we know that there exists an inequality $c_l x \leq d_l$ that is satisfied at equality by all integer solutions to $Ax=b, Cx \leq d$. This is (iii) in the statement of the theorem.
\end{proof}

\subsection{Chubanov's proof of Theorem~\ref{thm:0-1sol}}
Chubanov is able to convert the existential results of Propositions~\ref{prop:eq-forced} and \ref{prop:thin-strip} into constructive procedures, wherein he can find the corresponding implied equalities in strongly polynomial time. He then proceeds to apply the D\&C procedure iteratively and reduce the number of inequalities by one at every iteration. One needs at most $l$ such iterations and at the end one is left with a system of equations. This system can be tested for feasibility in strongly polynomial time by standard linear algebraic procedures. This is the main idea behind the Chubanov Relaxation algorithm and the proof of Theorem~\ref{thm:0-1sol} that appears in~\cite{chubanov}.

\old{Our goal now is to take Chubanov's Relaxation Algorithm as the foundation to prove some new results, and develop a new feasibility algorithm for general LPs, one that does not rely on using Lemma 2.2 part (2) to make statements only about integer solutions.
}
%%%%%%%%%%%%%%%%%%%%%%%%%%%%%%%%%%%%%%%%%%%%%%%%%%%%%%%%%%%
%%%%%%%%%%%%%%%%%%%%%%%%%%%%%%%%%%%%%%%%%%%%%%%%%%%%%%%%%%%

\section{Linear feasibility problems with strictly feasible solutions} \label{strictlps}

We will now demonstrate that using the lemmas we proved in Section~\ref{basicchubanov}, we can actually give strongly polynomial time algorithms for certain classes of linear feasibility problems. More concretely, we aim to prove Theorem~\ref{thm:totally-unimodular}. Consider a linear feability problem in the following standard form.

\begin{equation} \label{eq:or-sys}
\begin{array}{l}
Ax = b, \\
-x \leq \textbf{0},
\end{array}
\end{equation}

where $A \in \R^{m \times n}$ and $b \in \R^m$. We will assume that the entries of $A$ and $b$ are integers and that $A$ has full row rank. Let $$\Delta_A = \max\{|det(B)| \st B \textrm{ is an } n\times n \textrm{ submatrix of }A\}$$ be the maximum subdeterminant of the matrix $A$. Let $P(A,b) = \{x \in \R^n \st Ax=b, -x \leq \textbf{0}\}$ denote the feasible region of \eqref{eq:or-sys}.

\begin{lemma}\label{claim:vertex_den}
Let $x$ be any vertex of $P(A,b)$. If $x_i > 0$ for some $i\in \{1, \ldots, n\}$, then $x_i \geq \frac{1}{\Delta_A}$.
\end{lemma}
\begin{proof}
Since $x$ is a vertex, there exists a nonsingular $n \times n$ submatrix $B$ of $A$ such that $x$ is the basic feasible solution corresponding to the basis $B$. That is, the non basic variables are $0$ and the basic variable values are given in the vector $B^{-1}b$. If for some $i$, $x_i > 0$ then this is a basic variable and therefore its value is at least $\frac{1}{|det(B)|}$ since $b$ is an integer vector. Since $|det(B)| \leq \Delta_A$, we have that $x_i \geq \frac{1}{\Delta_A}$.
\end{proof}

\begin{lemma}\label{lemma:strictly-feas}
Suppose we know that \eqref{eq:or-sys} has a strictly feasible solution, i.e. there exists $\bar x \in \R^n$ such that $A\bar x = b$ and $\bar x_i > 0$ for all $i \in \{1, \ldots, n\}$. If $P(A,b)$ is bounded, then there exists a solution $x^* \in \R^n$ such that $Ax^* = b$ and $x^*_i \geq \frac{1}{n\Delta_A}$ for all $i \in \{1, \ldots, n\}$.
\end{lemma}

\begin{proof}
Since we have a strictly feasible solution and $P(A,b)$ is a polytope, then for every $i \in \{1, \ldots, n\}$ there exists a vertex $\bar x^i$ of $P$ such that $\bar x^i_i > 0$. By Lemma~\ref{claim:vertex_den}, we have that $\bar x^i_i \geq \frac{1}{\Delta_A}$. Therefore, if we consider the solution $$x^*= \frac{1}{n}\sum_{i=1}^n\bar x^i,$$ i.e., the convex hull of all these $n$ vertices, then $x^*_i \geq \frac{1}{n\Delta_A}$ for all $i \in \{1, \ldots, n\}$.
\end{proof}

We now consider linear feasibility problems of the form \eqref{eq:or-sys} such that either it is infeasible or has a strictly feasible solution. We will present an algorithm to decide if such linear feasibility problems are feasible or infeasible. We call this algorithm the {\bf LFS Algorithm}, as an acronym for {\bf L}inear {\bf F}easibility problems with {\bf S}trict solutions. As part of the input, we will take $\Delta_A$, as well as a real number $r$ such that $P(A,b)\subset B(0, r)$, i.e., $\lVert x \rVert < r$ for all feasible solutions $x$. The running time of our algorithm will depend on $\Delta_A, r$, and $n$.
% We will assume that the standard form \eqref{eq:bddx-2} is given as input.

\begin{alg} {\small THE \textbf{LFS ALGORITHM}} \\
\underline{\emph{Input:}} $Ax = b$, $ -x \leq \textbf{0}$, such that either this system is infeasible, or there exists a \emph{strictly} feasible solution. We are also given a real number $r$ such that $P(A,b)\subset B(0, r)$, i.e., $\lVert x \rVert < r$ for all feasible solutions $x$. Moreover, we are given $\Delta_A$ as part of the input.\\
\underline{\emph{Output:}} A feasible point $x^*$ or the decision that the system is infeasible. \\
\textbf{Parameterize} \eqref{eq:or-sys}:
\begin{equation} \label{eq:dlp2-1}
\begin{array}{l}
Ax - bt = 0, \\
-x \leq 0,\\
-t \leq -1.
\end{array}
\end{equation}
\textbf{Run} Chubanov's D\&C subroutine on the strengthened version of \eqref{eq:dlp2-1}:
\begin{equation} \label{eq:dlp2-2}
\begin{array}{l}
Ax - bt = 0, \\
-x \leq -1, \\
-t \leq -2.
\end{array}
\end{equation}
with $z = \textbf{0}$, $\hat r = 2n\Delta_A \sqrt{r^2 + 1}$, and $\epsilon = 1$ \\
\textbf{If} Chubanov's D\&C subroutine finds an $\epsilon$-feasible solution $(x^*, t^*)$ \\
\hspace*{.2in} \textbf{Return} $\hat{x} = \frac{x^*}{t^*}$ as a feasible solution for \eqref{eq:or-sys}.\\
\textbf{Else} (Chubanov's D\&C subroutine fails or returns a separating hyperplane) \\
\hspace*{.2in} \textbf{Return} ``The system \eqref{eq:or-sys} is INFEASIBLE''
\end{alg}

\begin{thm}\label{thm:LFS-analysis}
The LFS Algorithm correctly determines a feasible point $x^*$ of \eqref{eq:or-sys} or determines the system is infeasible. The running time is $$O \left( m^3 + m^2n + n^2m + (2n\Delta_A \sqrt{r^2 + 1})^{\frac{1}{\log_2\left(\frac{7}{5}\right)}}\left(n^2+n\log\Delta_A + n\log(r)\right)\right)$$
\end{thm}

\begin{proof}
We first confirm the running time. We will use Proposition~\ref{prop:run-time}. Observe that for our input to the D\&C algorithm, $\lVert c_{max}\rVert = 1$, $\epsilon = 1$ and $\hat r = 2n\Delta_A \sqrt{r^2 + 1}$ and $N = n$. Therefore, $K = (2n\Delta_A \sqrt{r^2 + 1})^{\frac{1}{\log_2\left(\frac{7}{5}\right)}}$, $\rho = 0$ since we use the origin as the initial center for the D\&C algorithm, $\mu = \frac{1}{14n}$. Substituting this into \eqref{eq:run-time}, we get the stated running time for the LFS algorithm.

\old{Lemma 3.1 in \cite{chubanov} proves the algorithm finishes in
$$O \left(\hat{r}^{\frac{1}{\log_2\left(\frac{7}{5}\right)}} \right)$$
iterations since $c_{\max} = 1$ and $\epsilon = 1$. 
}
\bigskip
To prove the correctness of the algorithm, we need to look at each of the three cases Chubanov's D\&C can return.

\noindent CASE 1:  An $\epsilon$-approximate solution $(x^*,t^*)$ is found for \eqref{eq:dlp3}.  Then $(x^*,t^*)$ is an exact solution to \eqref{eq:dlp2-1}, and $\frac{x^*}{t^*}$ is a solution to \eqref{eq:or-sys}.

\noindent CASE 2:  The D\&C algorithm fails to complete, returning consecutive separating hyperplanes $(h_1,\delta_1)$ and $(h_2, \delta_2)$ such that $h_1 = -\gamma h_2$ for some $\gamma > 0$. Then Proposition~\ref{prop:eq-forced} says that there exists $i \in \{1, \ldots, n\}$ such that $x_i =0$ for all feasible solutions to \eqref{eq:or-sys}. But then the original system \eqref{eq:or-sys} has no strictly feasible solution, and by our assumption, is therefore actually infeasible.  Therefore if the D\&C fails, our original system $Ax = b$, $-x \leq \textbf{0}$ is infeasible. \\
CASE 3:  The final case is when the D\&C algorithm returns a separating hyperplane $(h,\delta)$, separating $B(0, \hat r)$ from the feasible set of \eqref{eq:dlp2-2}. We now show that this implies the original system $Ax = b$, $-x \leq \textbf{0}$ is infeasible. 

If not, then from our assumption, we have a strictly feasible solution $x^*$ for \eqref{eq:or-sys}. By Lemma~\ref{lemma:strictly-feas}, we know that $x^*_i \geq \frac{1}{n\Delta_A}$ for all $i \in \{1, \ldots, n\}$. Consider the point $(\bar x, \bar t)=(2n\Delta_Ax^*, 2n\Delta_A)$ in $\R^{n+1}$. We show that $(\bar x, \bar t)$ is feasible to \eqref{eq:dlp2-2}. Since $Ax^* = b$, we have that $A\bar x - b\bar t = 0$. Moreover, since $x^*_i \geq \frac{1}{n\Delta_A}$ for all $i \in \{1, \ldots, n\}$, we have that $\bar x = 2n\Delta_Ax^* \geq \textbf{1}$, i.e., $-\bar x \leq -1$. Finally, $\bar t = 2n\Delta_A \geq 2$, since $\Delta_A \geq 1$ and $n \geq 1$. Therefore, $-\bar t \leq -2$. Now we check the norm of this point $\lVert (\bar x, \bar t)\rVert = \lVert (2n\Delta_Ax^*, 2n\Delta_A)\rVert = 2n\Delta_A\lVert (x^*, 1)\rVert$. Since $x^* \in P(A,b) \subset B(0,r)$, we know that $\lVert x^* \rVert < r$. Therefore, $\lVert (\bar x, \bar t)\rVert < \hat r$. So $(\bar x, \bar t)$ is a feasible solution to \eqref{eq:dlp2-2} and $(\bar x, \bar t) \in B(0, \hat r)$. But D\&C returned a separating hyperplane separating $B(0,\hat r)$ from the feasible set of \eqref{eq:dlp2-2}. This is a contradiction. Hence, we conclude that $Ax = b$, $-x \leq \textbf{0}$ is infeasible. \end{proof}

\begin{cor}\label{cor:bounded-lp}
Consider the following system.
\begin{equation} \label{eq:bddx}
\begin{array}{l}
Ax = b, \\
\textbf{0} \leq x \leq \lambda \textbf{1}.
\end{array}
\end{equation}
Suppose that we know that the above system is either infeasible, or has a strictly feasible solution, i.e., there exists $\bar x$ such that $A\bar x = b$ and $\textbf{0} < \bar x < \lambda\textbf{1}$. Then there exists an algorithm which either returns a feasible solution to \eqref{eq:bddx}, or correctly decides that the system is infeasible, with running time $$O \left(m^3 + m^2n + n^2m + (2n\Delta_A \lambda\sqrt{2n + 1})^{\frac{1}{\log_2\left(\frac{7}{5}\right)}}\left(n^2 + n\log\Delta_A + n\log(\lambda)\right)\right).$$
\end{cor}

\begin{proof}

We first put \eqref{eq:bddx} a standard form.

\begin{equation} \label{eq:bddx-2}
\begin{array}{l}
Ax = b, \\
x + y = \lambda \textbf{1}, \\
-x \leq \textbf{0}, \\
-y \leq \textbf{0}.
\end{array}
\end{equation}

We will use the constraint matrix of \eqref{eq:bddx-2} :
\[
\tilde A = \left[ \begin{array}{cc} A & \mathbf{0}_{m \times n} \\ I_n & I_n\end{array}    \right],
\]
where $\mathbf{0}_{m \times n}$ denotes the $m\times n$ matrix with all 0 entires, and $I_n$ is the $n\times n$ identity matrix. Therefore, $P(\tilde A, [b, \lambda\textbf{1}])$ is the feasible set for \eqref{eq:bddx-2}. Also, note that $\Delta_{\tilde A} = \Delta_A$. Since $\textbf{0} \leq x, y \leq \lambda\textbf{1}$ for all $(x,y)\in P(\tilde A, [b, \lambda\textbf{1}])$, we know that $P(\tilde A, [b, \textbf{1}]) \subset B(0, \lambda\sqrt{2n})$. Therefore, we run LFS with $r= \lambda\sqrt{2n}$ and the system \eqref{eq:bddx-2} as input. Observe that since \eqref{eq:bddx} has a strictly feasible solution, so does \eqref{eq:bddx-2}. By Theorem~\ref{thm:LFS-analysis}, LFS either returns a feasible solution to \eqref{eq:bddx-2} which immediately gives a feasible solution to \eqref{eq:bddx}, or correctly decides that \eqref{eq:bddx-2} is infeasible and hence \eqref{eq:bddx} is infeasible. Moreover, the running time for LFS is $$\begin{array}{c} O \left((m^3 + m^2n + n^2m + (2n\Delta \sqrt{r^2 + 1})^{\frac{1}{\log_2\left(\frac{7}{5}\right)}}\left(n^2 + n\log\Delta_A + n\log(r)\right)\right) \\ = O \left(m^3 + m^2n + n^2m + (2n\Delta_A \lambda\sqrt{2n + 1})^{\frac{1}{\log_2\left(\frac{7}{5}\right)}}\left(n^2 + n\log\Delta_A + n\log(\lambda)\right)\right).\end{array}$$ \end{proof}

Corollary~\ref{cor:bounded-lp} is related to the following theorem of Schrijver. Let $\tilde \Delta_A = \max\{|det(B^{-1})| \st B \textrm{ is a nonsingular submatrix of A} \}$.

\begin{thm}[Theorem 12.3 in~\cite{Schrijver1986}]
A combination of the relaxation method and the simultaneous diophantine approximation method solves a system $Ax \leq b$ of rational linear inequalities in time polynomially bounded by $size(A,b)$ and by $\tilde \Delta_A$.
\end{thm}

On one hand, we have the additional assumptions of being bounded and having strictly feasible solutions. On the other hand, we can get rid of the dependence of the running time on $size(A,b)$ and the use of the simultaneous diophantine approximation method, which utilizes non-trivial lattice algorithms. It is also not immediately clear how $\Delta_A$ is related to $\tilde \Delta_A$. We finally prove Theorem~\ref{thm:totally-unimodular}.
\old{
\begin{cor} Consider the linear feasibility problem given by \eqref{eq:bddx}. Suppose $A$ is a totally unimodular matrix and $\lambda$ is bounded by a polynomial in $n, m$; the latter happens, for instance, when $\lambda=1$. Furthermore, suppose we know that if \eqref{eq:bddx} is feasible, it has a strictly feasible solution. Then there exists a strongly polynomial time algorithm that either finds a feasible solution of \eqref{eq:bddx}, or correctly decides that the system is infeasible. The running time for this algorithm is $O \left(m^3 + m^2n + n^2m + n^2(2n\lambda\sqrt{2n + 1})^{\frac{1}{\log_2\left(\frac{7}{5}\right)}}\right)$. If $\lambda=1$, then the running time can be upper bounded by $O (m^3 + m^2n + n^2m + n^{5.1})$.
\end{cor}
}
\begin{proof}[Proof of Theorem~\ref{thm:totally-unimodular}]
If $A$ is totally unimodular, then $\Delta_A = 1$. The result now follows from Corollary~\ref{cor:bounded-lp}.\end{proof}

%%%%%%%%%%%%%%%%%%%%%%%%%%%%%%%%%%%%%%%%%%%%%%%%%%%%%%%%%%%
%%%%%%%%%%%%%%%%%%%%%%%%%%%%%%%%%%%%%%%%%%%%%%%%%%%%%%%%%%%

\section{A General Algorithm for Linear Feasibility Problems} \label{LFG}

In this section, we describe an algorithm for solving general linear feasibility problems using Chubanov's D\&C algorithm and the ideas developed in Section~\ref{basicchubanov}. We call this algorithm the {\bf LFG Algorithm}, as an acronym for {\bf L}inear {\bf F}easibility problems in {\bf G}eneral.
\old{
Our proposed algorithm is a simplification of the algorithm proposed by Chubanov \cite{chubanov} that gives us a polynomial algorithm that decides the feasibility of the general LP system \eqref{eq:orsys}.  This is an improvement over the Chubanov Relaxation algorithm, which, as we mention in the previous section in more detail, does not truly decide feasibility.  Where Chubanov goes from being an LP feasibility algorithm to something in between LP feasibility and IP feasibility is his inability, when the D\&C returns a separating hyperplane, to determine if any inequality is really an implicit equlity.  \par

Instead of trying to directly determine which inequality is an implicit equality, we prove in Lemma 3.1 simply that such an equality must exist.  As we will see, this existential proof is all we need.

\begin{lemma}\label{lem:equality-forcing}
If the system $Ax - bt = 0, Cx - dt \leq -1, -t \leq -2$ is infeasible, then there exists $l \in \{1, \ldots, k\}$ such that $c_l\cdot x = d_l$ for all $x$ satisfying $Ax=b, Cx \leq d$.
\end{lemma}

\begin{proof}
We prove the contrapositive. So suppose that for all $k \in \{1, \ldots, l\}$, there exists $x_k$ satisfying $Ax_k=b, Cx_k \leq d$ with $c_k\cdot x_k < d_k$. Then using $\bar x = \frac{1}{l}(x_1 + \ldots + x_l)$, we get that $A\bar x = b$ and $c_k\bar x < d_k$ for all $k \in \{1, \ldots, l\}$. Let $\eta_k = d_k - c_k\cdot\bar x > 0$ and let $\eta = \min\{\frac{1}{2}, \eta_1, \ldots, \eta_l\}$. Therefore, $\eta > 0$. Let $x^* = \frac{\bar x}{\eta}$ and $t^* = \frac{1}{\eta}$. Then 
$$Ax^* - bt^* = \frac{1}{\eta}(A\bar x - b) = 0.$$
For every $k \in \{1, \ldots, l\}$, 
$$d_k t^* - c_k x^* = \frac{1}{\eta}(d_k - c_k \bar x) = \frac{\eta_k}{\eta}\geq 1.$$ 
Therefore, $c_k x^* - d_k t^* \leq -1$ for every $k \in \{1, \ldots, l\}$. Finally, since $t = \frac{1}{\eta} \geq 2$ be definition of $\eta$. And hence $-t \leq -2$.
\end{proof}

\begin{lemma}
If the strengthened parameterized system (2.) is infeasible, then either:
\begin{enumerate}[(i)]
\item The original system (2.) is infeasible or
\item $\exists$ $l$ such that $c_l x - d_l = 0$ for all $x \in P$.
\end{enumerate}
\end{lemma}

\begin{proof}
We can transform \eqref{eq:orsys} into
\begin{equation} \label{eq:eqls}
\begin{array}{l}
Ax = b, \\
C(x^+ - x^-) + s = d, \\
x^+, \ x^-, \ s \geq 0.
\end{array}
\end{equation}

Parameterizing this system and then strengthening it, as we did in the previous section, gives us
\begin{equation} \label{eq:streqls}
\begin{array}{l}
A'x' - b't = 0, \\
x' \leq -1, \\
-t \leq -2.
\end{array}
\end{equation}

Now let us assume \eqref{eq:streqls} is infeasible and \eqref{eq:eqls} is feasible (if and only if \eqref{eq:orsys} is feasible).  Also assume there does not exist an $l$ such that $x_l = 0$.  This implies $x_i < 0$ for all $i$.  Let 
$$\epsilon = \min \{\frac{1}{2}, -x_i\},$$
and define 
$$\hat{x} = \frac{x'}{\epsilon}, \ t = \frac{1}{\epsilon}.$$
Then we have
$$A' \hat{x} - bt = 0 \ \textrm{and} \ \hat{x} \leq -1$$
and so $(\hat{x},t)$ is a feasible solution for \eqref{eq:streqls}, contradicting our assumption the strengthened system was infeasible.  Thus there must exist some $l$ such that $x_l = 0$.
\end{proof}
}
Before we can state our algorithm and prove its correctness, we need a couple of other pieces. \old{ First, from Schrijver's Corollary 10.2b \cite{sch}, given a system \eqref{eq:orsys} we will be able to construct a radius $\hat{r}$ that is guaranteed to contain some solution if \eqref{eq:orsys} is feasible.}

\begin{lemma}\label{lem:schrijver-bound}[Schrijver Corollary 10.2b]
Let $P$ be a rational polyhedron in $\R^n$ of facet complexity $\phi$.  Define
$$Q = P \cap \{x \in \R^n \vert -2^{5n^2\phi} \leq x_i \leq 2^{5n^2\phi} \ \textrm{for} \ i = 1,\ldots,n\}.$$
Then dim$(P)$ = dim$(Q)$.
\end{lemma}

Thus, if we use $\hat{r} = 2^{5n^2\phi} \sqrt{n}$ then $B(\textbf{0},\hat{r})$ will circumscribe the hypercube in the above lemma from Schrijver. We will also have dim$(P)$ = dim$(P \cap B(\textbf{0},\hat{r}))$. \par
The second piece we need comes from Papadimitriou and Stieglitz \cite{papa}.  Simply put, the theorem states that \eqref{eq:orsys} is feasible if and only if some other system is strictly feasible (i.e. all the inequalities are strict inequalities).  The lemma stated below is a stronger version of their Lemma 8.7 in \cite{papa}.

\begin{lemma}\label{lem:papa}[Lemma 8.7 in \cite{papa}]
The system \eqref{eq:orsys} is feasible if and only if
\begin{equation} \label{eq:strict}
\begin{array}{l}
Ax = b, \\
Cx < d + \nu. \\
\end{array}
\end{equation}
is feasible, where $\nu = 2^{-2T}$ when $T$ is the size of the input data. Moreover, given a solution to~\eqref{eq:strict}, we can construct a solution to~\eqref{eq:orsys} in strongly polynomial time.
\end{lemma}
\old{
\begin{proof}
The direction from \eqref{eq:orsys} to \eqref{eq:strict} is obvious. \par
Let $a_i$ denote the $i$-th row of $A$ and $c_l$ denote the $l$-th row of $C$.  Given $x_0$ a solution of \eqref{eq:strict}, define $L_{x_0} = \{ c_l \ \vert \ d_l \leq c_l x_0 < d_l + \nu \}$.  Then either for every $c_k \notin L_{x_0}$, we can write
\begin{equation} \label{eq:dpscombo}
c_k = \sum_{c_l \in L_{x_0}} \beta_{k,l} c_l + \sum_{i = 1}^m \alpha_{k,i} a_i
\end{equation}
or we can construct a different solution $z_0$ of \eqref{eq:strict} such that $L_{x_0} \subset L_{z_0}$. \par
To show this claim, assume there exists some $k$ such that $c_k$ cannot be expressed as \eqref{eq:dpscombo}.  Let $L_{x_0}'$ and $A'$ be a maximal linearly independent subsets of $L_{x_0}$ and $A$.  Then consider the system
\begin{equation} \label{eq:dpsclm}
\begin{array}{l}
a_i z = 0 \ \ a_i \in A' \\
c_l z = 0 \ \ c_l \in L_{x_0}' \\
c_k z = 1.
\end{array}
\end{equation}
This is a linear system of full rank and thus has a solution $z'$.  Let $z_0 = x_0 + \lambda z'$ where $d_k - c_k x_0 < \lambda < d_k + \nu - c_k x_0$.  Then $L_{z_0}$ will contain not only $c_k$, but also every element in $L_{x_0}$ as well.  Thus the claim is proven: either every $c_k \notin L_{x_0}$ is a linear combination of $c_l's \in L_{x_0}$ and $a_i's \in A$, or we can construct $L_{z_0}$ that contains both $c_k$ and $L_{x_0}$. \par
Now we need to construct a solution of \eqref{eq:orsys} from $x_0$.  By the above construction of $L_{z_0}$ from $L_{x_0}$, let $L$ be a maximal set of this type.  As before, select maximal linearly independent subsets of $L$ and $A$, $L'$ and $A'$ respectively.  Then every $c_k \notin L'$ can be expressed as $\eqref{eq:dpscombo}$, and by Cramer's Rule we know the $\alpha_i's$ and $\beta_l's$ are numbers expressible from determinants of absolute value less than $2^T$: $\alpha_{k,i} = \frac{D_{k,i}}{|D|}$ and $\beta_{k,l} = \frac{D_{k,l}}{|D|}$.  Let $\hat{x}$ be a solution to
$$ \begin{array}{ll}
a_i x = b_i & i \in A' \\
c_l x = d_l & l \in L'.
\end{array} $$
Then for each $k$ we have
\begin{eqnarray*}
|D|(c_k \hat{x} - d_k) &=& \sum_{a_i \in A'} |D| (\alpha_{k,i} a_i \hat{x}) + \sum_{c_l \in I'} |D| (\beta_{l,i} c_l \hat{x}) - |D| d_k \\
&=& \sum_{\{i \vert a_i \in A'\}} |D| (\alpha_{k,i} b_i) + \sum_{\{l \vert c_l \in I'\}} |D| (\beta_{l,i} d_l) - |D| d_k \\
&=& \sum_{\{i \vert a_i \in A'\}} |D| (\alpha_{k,i} b_i) + \sum_{\{l \vert c_l \in I'\}} |D| (\beta_{l,i} d_l) - |D| (d_k - c_k x_0) - |D|c_k x_0 \\
&=& \sum_{\{i \vert a_i \in A'\}} |D| (\alpha_{k,i} b_i) + \sum_{\{l \vert c_l \in I'\}} |D| (\beta_{l,i} d_l) - \sum_{a_i \in A'} |D| (\alpha_{k,i} a_i x_0) - {} \\
&& {} - \sum_{c_l \in I'} |D| (\beta_{l,i} c_l x_0) - |D| (d_k - c_k x_0) \\
&=& \sum_{\{i \vert a_i \in A'\}} |D| (\alpha_{k,i} (b_i - a_i x_0)) + \sum_{\{l \vert c_l \in I'\}} |D| (\beta_{l,i} (d_l - c_l x_0)) + |D|(c_k x_0 - d_k \\
&\leq& 0 + |D|\sum_{\{l \vert c_l \in I'\}} |\beta_{l,i} (c_l x_0 - d_l)| + |D|\nu \\
&<& \sum_{\{l \vert c_l \in I'\}} |D_{k,l}|\nu + |D|\nu \\
&<& (\vert L' \vert + 1) (2^T) (2^{-2T}) \\
&<& 1.
\end{eqnarray*}
Thus for all $k$ we have $|D|(c_k \hat{x} - d_k) < 1$ even though this expression is an integer, since all components of $\hat{x}$ divide $|D|$.  Thus $c_k \hat{x} - d_k \leq 0$ and so $\hat{x}$ is in fact a solution to \eqref{eq:orsys}.
\end{proof}

\emph{How do we know the above $\hat{x}$ works for all of the equalities in the original system?}\\

Now for the version we will use in the feasibility algorithm.  The following corollary is just a special case of the above lemma.

\begin{cor}
The system
\begin{equation} \label{eq:dsys}
\begin{array}{l}
Ax = b, \\
-x \leq 0.
\end{array}
\end{equation}
is feasible if and only if
\begin{equation} \label{eq:dstr}
\begin{array}{l}
Ax = b, \\
-x < \nu.
\end{array}
\end{equation}
is feasible, where $\nu = 2^{-2T}$ and $T$ is the bit length of $A$ and $b$. Moreover, given a solution to~\eqref{eq:dstr}, we can construct a solution to~\eqref{eq:dsys} in strongly polynomial time.
\end{cor}

We are now ready to state our new feasibility algorithm.  There are just a couple preprocessing points to consider:
\begin{enumerate}
\item Given a general LP \eqref{eq:orsys}, we first need to introduce slack variables, giving us the system
\begin{equation} \label{eq:presys}
\begin{array}{l}
A(x^+ - x^-) = b, \\
C(x^+ - x^-) + s = d, \\
x^+, \ x^-, \ x \geq 0.
\end{array}
\end{equation}
\item In the algorithm we will only refer to \eqref{eq:dsys}, as we will assume the transformation to \eqref{eq:presys} has already taken place.
\end{enumerate}

\emph{Do we still need this preprocessing we talked about on Friday given Amitabh's updated Lemma 3.1 and the stronger version of the P-S Lemma 8.7 we proved?}
}
\begin{alg} {\small THE LFG ALGORITHM} \\
\emph{Input:} $Ax = b$, $-x \leq 0$. \\
\emph{Output:} A feasible point $x^*$ or the decision the system is infeasible. \\
\textbf{Set} $\nu = 2^{-2T}$ where $T$ is the bit length of the input data\\
\textbf{Set} a new system
\begin{equation} \label{eq:dlp1}
\begin{array}{l}
Ax = b, \\
Cx \leq d + \frac{\nu}{2}.
\end{array}
\end{equation}
\textbf{Parameterize} \eqref{eq:dlp1}:
\begin{equation} \label{eq:dlp2}
\begin{array}{l}
Ax - bt = 0, \\
Cx - (d+\frac{\nu}{2}) t \leq 0,\\
-t \leq -1.
\end{array}
\end{equation}
\textbf{Run} Chubanov's D\&C subroutine on the strengthened version of \eqref{eq:dlp2}
\begin{equation} \label{eq:dlp3}
\begin{array}{l}
Ax - bt = 0, \\
Cx - (d+\frac{\nu}{2}) t \leq -1,\\
-t \leq -2.
\end{array}
\end{equation}
with $z = \textbf{0}$, $r = \hat{r}$ from Lemma 3.2 applied to \eqref{eq:dlp3}, and $\epsilon = 1$ \\
\textbf{If} a feasible solution $(x^*, t^*)$ is found \\
\hspace*{.2in} \textbf{Return} $\hat{x}$ from the proof of Lemma 3.3, with $x_0 = \frac{x^*}{t^*}$ \\
\textbf{Else} Lemma 3.1 implies one of our inequalities is an implied equality \\
\hspace*{.2in} \textbf{Return} ``The system \eqref{eq:dlp1} is INFEASIBLE''
\end{alg}

\begin{thm}
The LFG Algorithm correctly determines a feasible point $x^*$ of \eqref{eq:orsys} or determines the system is infeasible in a finite number of steps.
\end{thm}

\begin{proof}
\old{The only piece we need to actually worry about finishing in a finite number of steps is Chubanov's D\&C subroutine.  But Lemma 3.1 in \cite{chubanov} proves the algorithm finishes in
$$O \left( \left( \hat{r}(1 + \nu)\right)^{\frac{1}{\log_2\left(\frac{7}{5}\right)}} \right)$$
iterations.  Thus we know the LFG Algorithm finishes in a finite number of steps.}
To prove the correctness of the algorithm, we need to look at each of the three cases Chubanov's D\&C can return. \\
CASE 1:  An $\epsilon$-approximate solution $(x^*,t^*)$ is found for \eqref{eq:dlp3}.  Then $(x^*,t^*)$ is an exact solution to \eqref{eq:dlp2}, and $\frac{x^*}{t^*}$ is a solution to \eqref{eq:dlp1}, and hence is also a solution to \eqref{eq:strict}. Using Lemma~\ref{lem:papa}, we can construct a feasible solution to our original system \eqref{eq:orsys}. \\
CASE 2:  The D\&C algorithm fails to complete, returning consecutive separating hyperplanes $(h_1,\delta_1)$ and $(h_2, \delta_2)$ such that $h_1 = -\gamma h_2$ for some $\gamma > 0$. By Proposition~\ref{prop:eq-forced}, we know that there exists $k \in \{1, \ldots, l\}$ such that $c_kx = d_l + \frac{\nu}{2}$ for all solutions to $Ax=b, Cx \leq d + \frac{\nu}{2}$. But this simply implies that $Ax=b, Cx \leq d$ is infeasible.\old{ Then we have two possibilities to consider.  Either \eqref{eq:dlp1} is infeasible, and so the original system \eqref{eq:dsys} is also infeasible, or we can apply Chubanov's Lemma 4.1 \cite{chubanov} and say that for all feasible $x$ of \eqref{eq:dlp1}, there exists at least one $l$ such that $-x_l = \frac{\nu}{2}$.  But then $-x_l = \frac{\nu}{2} > 0$, and thus the original system \eqref{eq:dsys} is infeasible.  Therefore if the D\&C fails, our original system $Ax = b$, $-x \leq 0$ is infeasible.} \\
CASE 3:  The final case is when the D\&C returns a separating hyperplane $(h,\delta)$.  Note that due to the $\hat{r}$ we use, this already implies \eqref{eq:dlp3} is infeasible.  Then by Lemma~\ref{lem:equality-forcing1}, we know there exists some $l$ such that $c_lx = d_l + \frac{\nu}{2}$ for all $x$ satisfying $Ax=b, Cx \leq d + \frac{\nu}{2}$. Again, as in Case 2, this implies that $Ax = b$, $Cx \leq d$ is infeasible.\end{proof}

Since the running time of the D\&C subroutine is a polynomial in $\hat{r}$, and $\hat{r}$ is exponential in the input data, our algorithm is not guaranteed to run in polynomial time. However, as mentioned in the Introduction, it has certain advantages over Chubanov's polynomial time LP algorithm from~\cite{chubanovlp}. Firstly, it avoids the complicated reformulations used by Chubanov, which can potentially cause the actual runtime of his algorithm to be bad in practice. Secondly, the Chubanov Relaxation algorithm requires multiple iterations of the D\&C subroutine, whereas the LFG algorithm uses only one application of the D\&C subroutine. \old{This might again lead to a better runtime in practice.}

\old{
\begin{cor}
If \eqref{eq:orsys} takes the form
\begin{equation} \label{eq:bddx}
\begin{array}{l}
Ax = b \\
\textbf{0} \leq x \leq \lambda \textbf{1}
\end{array}
\end{equation}
then the LFG Algorithm runs in strongly polynomial time.
\end{cor}

\begin{proof}

\end{proof}
}

%%%%%%%%%%%%%%%%%%%%%%%%%%%%%%

\section{Computational Experiments} \label{experiments}

In this final section, we investigate the computational performance of the various relaxation methods mentioned in the preceding sections.  We used MATLAB to implement the following algorithms.

\begin{enumerate}
\item {\em The Chubanov D\&C algorithm}, as described in Section~\ref{basicchubanov}. We use this algorithm on linear feasibility problems in the following way. We choose $z=0$, $\epsilon = 1 \times 10^{-6}$ and $r$ is taken as $\sqrt{n + 1}$ for binary problems, and it is taken as the input dependent bound given in Lemma~\ref{lem:schrijver-bound}. This would mean that when the algorithm terminates, we either have an $\epsilon$-approximate solution, or we conclude that the system is infeasible.
\item {\em The Chubanov Relaxation algorithm}, as described in~\cite{chubanov}. Apart from the input data $A, b, C, d$, this algorithm requires as input a real number $r$ such that the feasible region is contained in $B(0,r)$. As with the Chubanov D\&C algorithm above, $r$ is taken as $\sqrt{n + 1}$ for binary problems, and it is taken as the input dependent bound given in Lemma~\ref{lem:schrijver-bound} for general linear feasibility.
\item {\em The LFS algorithm}, as described in Section~\ref{strictlps}. This algorithm also requires  as input a real number $r$ such that the feasible region is contained in $B(0,r)$. As with the Chubanov D\&C algorithm above, $r$ is taken as $\sqrt{n + 1}$ for binary problems, and it is taken as the input dependent bound given in Lemma~\ref{lem:schrijver-bound} for general linear feasibility problems. Moreover, we require as input $\Delta_A$, the maximum subdeterminant of the matrix $A$. We use the standard bound $\Delta_A \leq n^{\frac{n}{2}}|a_{max}|^n$, where $a_{max}$ is the entry of $A$ with largest absolute value.
\item Two versions of the original relaxation algorithm developed by Agmon \cite{agmon}, and Motzkin and Schoenberg \cite{motzkinschoenberg}.
\end{enumerate}

Recall that the Chubanov Relaxation algorithm is not really a linear feasibility algorithm; it may sometimes report that the solution has no integer solutions (see the statement of Theorem~\ref{thm:0-1sol}). However, we feel it is still interesting to study its practical run time, and compare it with our linear feasibility algorithms which are also based on the Chubanov D\&C algorithm. 

It has already been shown that, for most purposes, the original relaxation method is not able to compete with other linear programming algorithms \cite{goffinnonpoly,telgen}. Thus, the purpose of these experiments is not to compare the running times with current commercial software.  Rather, we want to determine how the new relaxation-type algorithms, based on the Chubanov D\&C algorithm, compare with the original relaxation algorithm suggested by Agmon, and Motzkin and Schoenberg.  Despite the improved theoretical performance of the new algorithms, the tables below show that, in practice, the original relaxation method is by far the preferred method in almost every case.

As noted above, the algorithms and experiment scripts were all developed in MATLAB 7.12.0.  No parallelization was incorporated into the algorithms.  The computational experiments were run on a personal computer with an Intel Core i5 M560 2.67 GHz processor.  The problems used were drawn from the Netlib repository \cite{Netlib}, the MIPLIB repository \cite{BixbyCeriaMcZealSavelsbergh1998,KochEtAl2011}, Hoffman's experiments \cite{hoffman}, Telgen's and Goffin's example \cite{goffinnonpoly,telgen} which shows the exponential behavior of the original relaxation method, and some randomly generated problems. The code and the problems used are available at {\tt http://www.math.ucdavis.edu/$\sim$mjunod/}.

In every table, a dash ``--'' denotes an experiment that exceeded our default time limit of 10 minutes.  For example, in Table \ref{tab:ChubAndLFS}  Chubanov's algorithm timed out on the third Telgen experiment.  Given the size and complexity of the problems involved, we determined that any algorithm exceeding 10 minutes had already shown how practically inefficient it was for that problem. \old{In cases where a particular algorithm timed out on all the test problems, the column is not shown.  For example, the LFS results are not shown in Tables \ref{tab:ChubAndLFS} and \ref{tab:random0/1problems} because all the experiments timed out.  Similarly, }When all of the algorithms timed out on a single problem, the results for the problem are not reported in the tables. \old{Every experiment in the tables uses the origin as its starting point. When necessary, we start experiments with a radius derived from the Schrijver bound (Corollary 10.2b in \cite{Schrijver1986}), or we let $r = \lceil \sqrt{n+1} \rceil$ when dealing with a binary problem.  However, in some of the non-binary problems the Schrijver bound becomes intractable in MATLAB.  Thus for }We note here that for some experiments we rely on some familiarity with the problem and determine our own bound for $r$ in the algorithms, to speed up the computations. \old{Significantly larger radii will lead to longer run times of the D\&C based algorithms.  }All of the run times reported are in seconds.

We note here that our LFS algorithm timed out on all instances tried. The Chubanov Relaxation algorithm and the Chubanov D\&C algorithms timed out on all Netlib and MIPLIB problems, as well as on all the problems from Hoffman's experiments. Hence, the LFS algorithm is not reported in any table. Results for the Chubanov Relaxation algorithm and the Chubanov D\&C algorithms on the Telgen examples are reported in Table~\ref{tab:ChubAndLFS}, and their results on the random $0-1$ instances are reported in Table~\ref{tab:random0/1problems}.

\begin{table}[htp] 
\caption{Telgen Results for the Chubanov Relaxation Algorithm and the Chubanov D\&C Algorithm}  \label{tab:ChubAndLFS}
\centering
	\begin{tabular}{| l | c | c | c | c |}
	\cline{2-5}
	\multicolumn{1}{c}{} & \multicolumn{2}{|c|}{Chubanov Relaxation} & \multicolumn{2}{|c|}{Chubanov D\&C} \\
	\hline
	\multirow{2}{*}{Experiment} & \multirow{2}{*}{Recursions} & Time & \multirow{2}{*}{Recursions} & Time \\
	&  & (Sec) &  & (Sec) \\
	\hline
	Telgen ($\alpha = 1$) & 139254 & 30.4761 & 51 & 0.0010 \\
	\hline
	Telgen ($\alpha = 2$) & $2.12991 \times 10^6$ & 451.817 & 109 & 0.0177 \\
	\hline
	Telgen ($\alpha = 3$)  & -- & -- & 116 & 0.0211\\
	\hline
	Telgen ($\alpha = 4$)  & -- & -- & 124 & 0.0169 \\
	\hline
	Telgen ($\alpha = 5$)  & -- & -- & 132 & 0.0218 \\
%	\hline
%	Hoffman (6D) & -- & -- & -- & 12+ hours so far \\
	\hline
	\end{tabular} 
\end{table}

Table \ref{tab:motzkin1} compares two variants of the original relaxation method.  The two versions of the original relaxation method differ only in how the violated constraint is chosen.  See \cite{agmon, motzkinschoenberg} for a full explanation of the algorithm.  In the first implementation, called ``Regular'' in the tables, we chose the maximally violated constraint as specified by Agmon, and Motzkin and Schoenberg.  Our second implementation, called ``Random'' in the tables, randomly chooses a violated constraint to see if there is any practical gain, as Needell's paper on the Kaczmarz method \cite{Needell} suggests might be possible.  Every time we ran the ``Random'' version on a problem, we ran it 100 times and we are reporting the average number of iterations, time in seconds, and the standard deviation of each data set.  In the experiments labeled Telgen, only two constraints exist in the problem and only one is violated at any iteration.  Thus only the ``Regular'' version is reported as the two gave identical results.  For every experiment we set $\lambda = 1.9$ as a higher over-projection constant increases the speed of convergence, and let $\epsilon = 1 \times 10^{-6}$ be our error constant.

Table \ref{tab:random0/1problems} compares the performance of the Chubanov Relaxation algorithm, the Chubanov D\&C algorithm and the original relaxation algorithm, on the randomly generated 0-1 problem set.  These were problems of the form $Ax = b$, $\zero \leq x \leq \one$ with the dimension noted in each row.  We limited ourselves to a randomly generated 0-1 matrix $A$ with anywhere from 1 to $n-1$ rows, also randomly chosen, and populated $b$ with integers randomly chosen from the set $\{1,\ldots,n\}$. \old{Our goal here was to generate problems with a reasonable likelihood of being feasible and with a feasible region that could intersect the 0-1 cube formed by the inequality constraints at any corner. }For each dimension, we generated 10 random problems and then reported the average behavior along with the standard deviation.  Note that for the D\&C algorithm, the high standard deviations indicate that for the vast majority of the problems it ran quickly.

\old{Together Tables \ref{tab:motzkin1} and \ref{tab:random0/1problems} list most of the problems used in the experiments.  Table \ref{tab:ChubAndLFS} is shortened because the D\&C algorithms timed out on all of the additional problems listed in Table \ref{tab:motzkin1}.  Taken all together, these results show that the new D\&C based relaxation algorithms cannot compete practically with the original relaxation algorithm.}

Despite what appears to be the reasonable performance of the Chubanov D\&C algorithm with the Random 0-1 problems, both the Chubanov Relaxation algorithm and the LFS algorithm performed much worse (in fact, the LFS algorithm timed out on all instances). Ironically, it is actually the D\&C subroutine in these algorithms that causes this. Indeed, when we strengthen and homogenize the linear system, we greatly increase the parameter $r$ that is used by the D\&C subroutine in these two algorithms, creating \old{the need to run through not just several more recursions -- as increasing $r$ and potentially $\|c_{max}\|$ only increases the depth of the recursion tree by a logarithmic amount -- but }a very large number of new nodes that are added to the recursion tree of the D\&C algorithm.  \old{These additional nodes are what really slow down the algorithms, as does the need of Chubanov's algorithm to run the D\&C on this increased number of nodes multiple times.  }This creates a significant increase in the run times of these new algorithms, and as a result, they cannot compete practically.

\begin{sidewaystable}[htp] 
\caption{Test Results for the Original Relaxation Algorithms} \label{tab:motzkin1}
\centering
	\begin{tabular}[c]{| l | c | c | c | c | c | c |}
	\cline{2-7}
	\multicolumn{1}{c}{} & \multicolumn{2}{|c|}{Regular} & \multicolumn{4}{|c|}{Random} \\
	\hline
	\multirow{2}{*}{Experiment} & \multirow{2}{*}{Iterations} & \multirow{2}{*}{Time (Sec)} & \multicolumn{2}{|c|}{Iterations} & \multicolumn{2}{|c|}{Time (Sec)} \\
	\cline{4-7}
	& & & Avg/Std Dev & Min/Max & Avg/Std Dev & Min/Max \\
	\hline
	Telgen ($\alpha = 1$) & 7 & 0.0077 & \multicolumn{4}{|c|}{N/A} \\
	\hline
	Telgen ($\alpha = 2$) & 14 & 0.0013 & \multicolumn{4}{|c|}{N/A} \\
	\hline
	Telgen ($\alpha = 3$) & 28 & 0.0033 & \multicolumn{4}{|c|}{N/A} \\
	\hline
	Telgen ($\alpha = 4$) & 94 & 0.0073 & \multicolumn{4}{|c|}{N/A} \\
	\hline
	Telgen ($\alpha = 5$) & 2153 & 0.0991 & \multicolumn{4}{|c|}{N/A} \\
	\hline
	ADLITTLE & 1774 & 0.29 & -- & -- & -- & -- \\
	\hline
	AFIRO & 1018 & 0.0795 & 948/0 & 948/948 & 0.1102/0.0052 & 0.1082/0.1583 \\
	\hline
	BEACONFD & 882 & 0.3296 & -- & -- & -- & -- \\
	\hline
	BLEND & 56241 & 7.4636 & 4783/0 & 4783/4783 & 1.1451/0.0179 & 1.1198/1.2377  \\
	\hline
	E226 & $1.01592 \times 10^6$ & 490.399 & -- & -- & -- & -- \\
	\hline
	RECIPELP & 11008 & 0.7245 & -- & -- & -- & -- \\
	\hline
	SC50A & 24 & 0.0245 & 137/0 & 137/137 & 0.0373/0.0034 & 0.0239/0.0587  \\
	\hline
	SC50B & 9 & 0.0194 & 86/0 & 86/86 & 0.0321/0.0028 & 0.0178/0.0422 \\
	\hline
	SC105 & 504 & 0.1133 & 856/0 & 856/856 & 0.1238/0.0094 & 0.1078/0.2107 \\
	\hline
	SCAGR7 & 41716 & 8.377 & 26144/0 & 26144/26144 & 4.8554/0.0507 & 4.7997/5.0119 \\
	\hline
	SHARE2B & 591986 & 58.7718 & 18469/0 & 18469/18469 & 5.8692/0.0316 & 5.8274/6.0655 \\
	\hline
	STOCFOR1 & $1.7042 \times 10^6$ & 236.971 & -- & -- & -- & -- \\
	\hline
	Hoffman (6D) & 8 & 0.0017 & 6.64/0.6594 & 5/7 & 0.0138/0.0015 & 0.0125/0.0164 \\
	\hline
	Hoffman (7D) & 11 & 0.0026 & 8.78/0.9383 & 6/10 & 0.0141/0.0012 & 0.0131/0.0185\\
	\hline
	Hoffman (8D) & 11 & 0.0019 & 10.04/2.6777 & 7/16 & 0.0139/0.0019 & 0.0045/0.0151 \\
	\hline
	Hoffman (9D) & 49 & 0.0045 & 61.32/6.6087 &  40/77 & 0.0179/0.0020 & 0.0148/0.0191 \\
	\hline
	Hoffman (10D) & 9982 & 0.3253 & -- & -- & -- & -- \\
	\hline
	
	\hline
	\end{tabular}
\end{sidewaystable}

\begin{table}[htp]
\caption{Test Results for Random Problems Bounded by the 0-1 Cube $[0,1]^n$} \label{tab:random0/1problems} 
\centering
	\begin{tabular}{| l | c | c | c | c | c | c |}
	\cline{2-7}
	\multicolumn{1}{c}{} & \multicolumn{2}{|c|}{Chubanov Relaxation} & \multicolumn{2}{|c|}{Chubanov D\&C} & \multicolumn{2}{|c|}{Original Relaxation}\\
	\hline
	\multirow{2}{*}{Experiment} & Recursions & Time (Sec) & Recursions & Time (Sec) & Iterations & Time (Sec)\\	
	& (Avg/SD) & (Avg/SD) & (Avg/SD) & (Avg/SD) & (Avg/SD) & (Avg/SD) \\
	\hline
	\multirow{2}{*}{Random 2D} & 61420/ & 12.309/ & 64.7/ & 0.0051/ & 0.5/ & 0.0004/ \\
	& 0 & 0.2771 & 21.679 & 0.0026 & 0.5 & 0.0003 \\
	\hline
	\multirow{2}{*}{Random 3D} & 151522/ & 30.057/ & 1258.3/ & 0.2122/ & 0.75/ & 0.0005/ \\
	& 0 & 0.2518 & 3799.59 & 0.6538 & 0.5 & 0.0002 \\
	\hline
	\multirow{2}{*}{Random 4D} & 520152/ & 102.52/ & 419657.3/ & 70.566/ & 0.6667/ & 0.0004/ \\
	& 0 & 0.7359 & $1.1198 \times 10^6$ & 187.59 & 0.5774 & 0.0003 \\
	\hline
	\multirow{2}{*}{Random 5D} & $1.012 \times 10^6$/ & 199.94/ & $1.3395 \times 10^6$/ & 247.76/ & 0.3333/ & 0.0002/ \\
	& 0 & 1.8034 & $1.6308 \times 10^6$ & 304.01 & 0.5774 & 0.0001 \\
	\hline
	\multirow{2}{*}{Random 6D} & $1.733 \times 10^6$/ & 344.36/ & 774996.6/ & 140.28/ & 20/ & 0.0015/ \\
	& 0 & 1.0829 & $1.63 \times 10^6$ & 250.41 & 26.870 & 0.0016 \\
	\hline
	\multirow{2}{*}{Random 7D} & $2.6938 \times 10^6$/ & 544.78/ & 772084.4/ & 120.01/ & \multirow{2}{*}{--} & \multirow{2}{*}{--} \\
	& 68925 & 19.641 & $1.6303 \times 10^6$ & 252.98 &  &  \\
	\hline
	\multirow{2}{*}{Random 8D} & \multirow{2}{*}{--} & \multirow{2}{*}{--} & 165.3/ & 0.0108/ & 0.6667/ & 0.0004/ \\
	&  &  & 355.41 & 0.0209 & 0.5774 & 0.0002 \\
	\hline
	\multirow{2}{*}{Random 9D} & \multirow{2}{*}{--} & \multirow{2}{*}{--} & 309896.5/ & 60.004/ & 1/ & 0.0004/ \\
	&  &  & 979808.4 & 189.74 & 0 & 0 \\
	\hline
	\multirow{2}{*}{Random 10D} & \multirow{2}{*}{--} & \multirow{2}{*}{--} & 32224.1/ & 2.0451/ & 239/ & 0.0127/ \\
	&  &  & 101037.4 & 6.3976 & 0 & 0 \\
	\hline
	\end{tabular}
\end{table}

%\begin{sidewaystable}[htp] \label{tab:MIPLIBproblems}
%\caption{Test Results for 0-1 MIPLIB Problems}
%\centering
%	\begin{tabular}{| l | c | c | c | c | c | c | c | c |}
%	\cline{2-9}
%	\multicolumn{1}{c}{} & \multicolumn{2}{|c|}{Chubanov} & \multicolumn{2}{|c|}{LFS} & \multicolumn{2}{|c|}{D\&C} & \multicolumn{2}{|c|}{Original Relaxation}\\
%	\hline
%	Experiment & Recursions & Time & Recursions & Time & Recursions & Time & Iterations & Time \\
%	\hline
%	enigma & -- & -- & -- & -- & 96 & 0.0670 & 1 & 0.0068 \\
%	\hline
%	lseu & -- & -- & -- & -- & -- & -- & 25 & 0.0209 \\
%	\hline
%	mod008 & -- & -- & -- & -- & -- & -- & 2 & 0.0048 \\
%	\hline
%	p0033 &  &  &  &  & -- & -- & 63 & 0.2134 \\
%	\hline
%	p0201 &  &  &  &  & -- & -- & 2860 & 0.4905 \\
%	\hline
%	p0282 &  &  &  &  & -- & -- & 219 & 0.0763 \\
%	\hline
%	p0548 &  &  &  &  & -- & -- & 770 & 1.3291 \\
%	\hline
%	stein27 &  &  &  &  & -- & -- & 2 & 0.0013 \\
%	\hline
%	stein45 &  &  &  &  & -- & -- & 2 & 0.0008 \\
%	\hline
%	\end{tabular}
%\end{sidewaystable}

\pagebreak

\bibliographystyle{plain}
\bibliography{bibliolp}

\begin{thebibliography}{10}

\bibitem{agmon}
S.~Agmon.
\newblock The relaxation method for linear inequalities.
\newblock {\em Canadian J. Math.}, 6:382--392, 1954.

\bibitem{amaldihauser}
E.~Amaldi and R.~Hauser.
\newblock Boundedness theorems for the relaxation method.
\newblock {\em Math. Oper. Res.}, 30(4):939--955, 2005.

\bibitem{perceptron}
A.~Belloni, R.~Freund, and S.~Vempala.
\newblock An efficient re-scaled perceptron algorithm for conic systems.
\newblock {\em Mathematics of Operations Research}, 34(3):621--641, 2009.

\bibitem{betkelp}
U.~Betke.
\newblock Relaxation, new combinatorial and polynomial algorithms for the
  linear feasibility problem.
\newblock {\em Discrete Comput. Geom.}, 32(3):317--338, 2004.

\bibitem{BixbyCeriaMcZealSavelsbergh1998}
R.~E. Bixby, S.~Ceria, C.~M. McZeal, and M.~W.~P Savelsbergh.
\newblock An updated mixed integer programming library: {MIPLIB} 3.0.
\newblock {\em Optima}, 58:12--15, 1998.

\bibitem{chubanov}
S.~Chubanov.
\newblock A strongly polynomial algorithm for linear systems having a binary
  solution.
\newblock {\em Mathematical Programming (to appear)}, pages 1--38, 2011.
\newblock 10.1007/s10107-011-0445-3.

\bibitem{chubanovlp}
S.~Chubanov.
\newblock A polynomial relaxation-type algorithm for linear programming,
  unpublished manuscript (2011).

\bibitem{Netlib}
D.~M. Gay.
\newblock Electronic mail distribution of linear programming test problems.
\newblock 13:10--12, 1985.

\bibitem{goffin}
J.-L. Goffin.
\newblock The relaxation method for solving systems of linear inequalities.
\newblock {\em Math. Oper. Res.}, 5(3):388--414, 1980.

\bibitem{goffinnonpoly}
J.-L. Goffin.
\newblock On the nonpolynomiality of the relaxation method for systems of
  linear inequalities.
\newblock {\em Math. Programming}, 22(1):93--103, 1982.

\bibitem{hoffman}
A.~Hoffman, M.~Mannos, D.~Sokolowsky, and N.~Wiegmann.
\newblock Computational experience in solving linear programs.
\newblock {\em Journal of the Society for Industrial and Applied Mathematics},
  1(1):pp. 17--33, 1953.

\bibitem{Kaczmarzoriginal}
S.~Kaczmarz.
\newblock Approximate solution of systems of linear equations.
\newblock {\em Internat. J. Control}, 57(6):1269--1271, 1993.
\newblock Translated from the German original of 1933.

\bibitem{KochEtAl2011}
Thorsten Koch, Tobias Achterberg, Erling Andersen, Oliver Bastert, Timo
  Berthold, Robert~E. Bixby, Emilie Danna, Gerald Gamrath, Ambros~M. Gleixner,
  Stefan Heinz, Andrea Lodi, Hans Mittelmann, Ted Ralphs, Domenico Salvagnin,
  Daniel~E. Steffy, and Kati Wolter.
\newblock {MIPLIB} 2010.
\newblock {\em Mathematical Programming Computation}, 3(2):103--163, 2011.

\bibitem{maurrasetal}
J.-F. Maurras, K.~Truemper, and M.~Akg{\"u}l.
\newblock Polynomial algorithms for a class of linear programs.
\newblock {\em Math. Programming}, 21(2):121--136, 1981.

\bibitem{motzkinschoenberg}
T.~S. Motzkin and I.~J. Schoenberg.
\newblock The relaxation method for linear inequalities.
\newblock {\em Canadian J. Math.}, 6:393--404, 1954.

\bibitem{Needell}
D.~Needell.
\newblock Randomized {K}aczmarz solver for noisy linear systems.
\newblock {\em BIT}, 50(2):395--403, 2010.

\bibitem{papa}
C.~H. Papadimitriou and K.~Steiglitz.
\newblock {\em Combinatorial Optimization : Algorithms and Complexity}.
\newblock Dover Books on Computer Science. Courier Dover Publications, 1998.

\bibitem{Schrijver1986}
A.~Schrijver.
\newblock {\em Theory of linear and integer programming}.
\newblock Wiley-Interscience Series in Discrete Mathematics. John Wiley \& Sons
  Ltd., 1986.
\newblock A Wiley-Interscience Publication.

\bibitem{strohmervershynin}
T.~Strohmer and R.~Vershynin.
\newblock A randomized {K}aczmarz algorithm with exponential convergence.
\newblock {\em J. Fourier Anal. Appl.}, 15(2):262--278, 2009.

\bibitem{tardos-strongly-poly}
E.~Tardos.
\newblock A strongly polynomial algorithm to solve combinatorial linear
  programs.
\newblock {\em Math. of Oper. Res.}, 34(2):250--256, 1986.

\bibitem{telgen}
J.~Telgen.
\newblock On relaxation methods for systems of linear inequalities.
\newblock {\em European J. Oper. Res.}, 9(2):184--189, 1982.

\bibitem{yevavasis}
S.~Vavasis and Y.~Ye.
\newblock A primal-dual interior point method whose running time depends only
  on the constraint matrix.
\newblock {\em Mathematical Programming}, 74:79--120, 1996.

\end{thebibliography}

\end{document}